\documentclass[aihp,preprint]{imsart}

\RequirePackage[OT1]{fontenc}
\RequirePackage{amsthm,amsmath}
\RequirePackage[numbers]{natbib}
\RequirePackage[colorlinks,citecolor=blue,urlcolor=blue]{hyperref}

\usepackage{amssymb}

\startlocaldefs
\numberwithin{equation}{section}
\theoremstyle{plain}
\newtheorem{theorem}{Theorem}[section]

\newtheorem{lemma}{Lemma}[section]
\newtheorem{corollary}{Corollary}[section]

\newtheorem{proposition}{Proposition}[section]
\newtheorem{assumption}[theorem]{Assumption}

\theoremstyle{definition}
\newtheorem{remark}[theorem]{Remark}
\newtheorem{example}[theorem]{Example}
\endlocaldefs

\usepackage{color}

\begin{document}

\begin{frontmatter}
\title{Stein's method for functions of multivariate normal random variables}
\runtitle{Functions of multivariate normal random variables}

\begin{aug}
\author{\fnms{Robert E.} \snm{Gaunt}\ead[label=e1]{robert.gaunt@manchester.ac.uk}}

\runauthor{R. E. Gaunt}

\affiliation{The University of Manchester\thanksmark{m1}}

\address{School of Mathematics\\
The University of Manchester\\
Manchester\\
M13 9PL\\
United Kingdom \\
\printead{e1}}
\end{aug}

\begin{abstract}
By the continuous mapping theorem, if a sequence of $d$-dimensional random vectors  $(\mathbf{W}_n)_{n\geq1}$ converges in distribution to a multivariate normal random variable $\Sigma^{1/2}\mathbf{Z}$, then the sequence of random variables $(g(\mathbf{W}_n))_{n\geq1}$ converges in distribution to $g(\Sigma^{1/2}\mathbf{Z})$ if $g:\mathbb{R}^d\rightarrow\mathbb{R}$ is continuous. In this paper, we develop Stein's method for the problem of deriving explicit bounds on the distance between $g(\mathbf{W}_n)$ and $g(\Sigma^{1/2}\mathbf{Z})$ with respect to smooth probability metrics.
We obtain several bounds for the case that the $j$-component of $\mathbf{W}_n$ is given by $W_{n,j}=\frac{1}{\sqrt{n}}\sum_{i=1}^nX_{ij}$, where the $X_{ij}$ are independent.  In particular, provided $g$ satisfies certain differentiability and growth rate conditions, we obtain an order $n^{-(p-1)/2}$ bound, for smooth test functions, if the first $p$ moments of the $X_{ij}$ agree with those of the normal distribution.  If $p$ is an even integer and $g$ is an even function, this convergence rate can be improved further to order $n^{-p/2}$.  These convergence rates are shown to be of optimal order.  We apply our general bounds to some examples, which include the distributional approximation of asymptotically chi-square distributed statistics; the approximation of expectations of smooth functions of binomial and Poisson random variables; rates of convergence in the delta method; and a quantitative variance-gamma approximation of the $D_2^*$ statistic for alignment-free sequence comparison in the case of binary sequences. 
\end{abstract}



\end{frontmatter}

\section{Introduction}

\subsection{Stein's method for multivariate normal approximation}

In 1972, Stein \cite{stein} introduced a powerful method that allows one to bound the distance between the distributions of a random variable $W$ and a standard normal random variable $Z$ with respect to a probability metric.  The basic approach, as described in detail in \cite{stein2} (see also \cite{chen} for a detailed introduction), involves two steps.  The first is to solve the so-called Stein equation
\begin{equation} \label{normal equation} f''(w)-wf'(w)=h(w)-\mathbb{E}h(Z),
\end{equation} 
where the test function $h$ is real-valued.  Bounds for the solution, $f$, and its derivatives are then established.  In the second step, the expectation 
\begin{equation}\label{normal e}\mathbb{E}[f''(W)-Wf'(W)]
\end{equation}
is bounded, typically through the use of coupling techniques, which, via (\ref{normal equation}), leads to a bound for the quantity of interest $\mathbb{E}h(W)-\mathbb{E}h(Z)$.  This then allows one to obtain bounds for the distance between the distributions of $W$ and $Z$ with respect to probability metrics of the form 
\begin{equation*}d_{\mathcal{H}}(\mathcal{L}(W),\mathcal{L}(Z))=\sup_{h\in\mathcal{H}}|\mathbb{E}h(W)-\mathbb{E}h(Z)|,
\end{equation*}
where the supremum is taken over a class of functions $\mathcal{H}$.  In the Stein's method literature, common classes of test functions include
\begin{align*}\mathcal{H}_{\mathrm{K}}&=\{\mathbf{1}(\cdot\leq z)\,|\,z\in\mathbb{R}\}, \\
\mathcal{H}_{\mathrm{W}}&=\{h:\mathbb{R}\rightarrow\mathbb{R}\,|\,\text{$h$ is Lipschitz, $\|h'\|\leq1$}\}, \\
\mathcal{H}_{p}&=\{h\in C^p(\mathbb{R})\,|\,\|h^{(k)}\|\leq1 \:\text{for all $1\leq k\leq p$}\},
\end{align*}
which give the Kolmogorov, Wasserstein and smooth Wasserstein (for $p\geq1$) distances, which we denote by $d_{\mathrm{K}}$, $d_{\mathrm{W}}$ and $d_{\mathcal{H}_p}$, respectively.  (Here and elsewhere in the paper $\|f\|:=\|f\|_{\infty}=\sup_{x\in\mathbb{R}}|f(x)|$.)  Other variants include restricting the smooth Wasserstein distance by further requiring that $\|h\|\leq1$ (see, for example, \cite{swan16}), or weakening the conditions to only require that $\|h^{(p)}\|\leq1$ (see, for example, \cite{f18}).

By recognising the left-hand side of the Stein equation (\ref{normal equation}) as the generator of an Ornstein-Uhlenbeck process, \cite{barbour2} and \cite{gotze} extended Stein's method for normal approximation to the multivariate normal distribution.  A generalisation of (\ref{normal equation}) to the multivariate normal distribution $\mathrm{MVN}(\mathbf{0},\Sigma)$ with mean $\mathbf{0}\in\mathbb{R}^d$ and covariance matrix $\Sigma\in\mathbb{R}^{d\times d}$ (see \cite{goldstein1}) is given by
\begin{equation} \label{mvn145} \nabla^T\Sigma\nabla f(\mathbf{w})-\mathbf{w}^T\nabla f(\mathbf{w})=h(\mathbf{w})-\mathbb{E}h(\Sigma^{1/2}\mathbf{Z}),
\end{equation}
where $\mathbf{Z}$ denotes a random vector having standard multivariate normal distribution of dimension $d$.  If $h:\mathbb{R}^d\rightarrow\mathbb{R}$ is Lipschitz, then a solution to (\ref{mvn145}) exists and is given by
\begin{equation}\label{mvnsoln}f(\mathbf{w})=-\int_{0}^{\infty}[\mathbb{E}h(\mathrm{e}^{-s}\mathbf{w}+\sqrt{1-\mathrm{e}^{-2s}}\Sigma^{1/2}\mathbf{Z})-\mathbb{E}h(\Sigma^{1/2}\mathbf{Z})]\,\mathrm{d}s
\end{equation}
(see \cite{fang18}, as well as \cite{barbour2} and \cite{kdv17} for an analagous solution for the Stein equation for approximation by Brownian motion on $[0,1]$). If $h$ is $k$ times differentiable, then the solution (\ref{mvnsoln}) satisfies the bound (see \cite{barbour2} and \cite{goldstein1}):
\begin{equation}\label{cheque} \bigg\|\frac{\partial^k f(\mathbf{w})}{\prod_{j=1}^k\partial w_{i_j}}\bigg\|\leq \frac{1}{k}\bigg\|\frac{\partial^k h(\mathbf{w})}{\prod_{j=1}^k\partial w_{i_j}}\bigg\|, \quad k\geq 1.
\end{equation}
 If we also suppose $\Sigma$ is positive definite, we can obtain a bound involving one fewer derivative of $h$ (see \cite{dobler thesis, gaunt rate}):
\begin{equation}\label{charm}\bigg\|\frac{\partial^k f(\mathbf{w})}{\prod_{j=1}^k\partial w_{i_j}}\bigg\|\leq\frac{\Gamma(\frac{k}{2})}{\sqrt{2}\Gamma(\frac{k+1}{2})}\min_{1\leq l\leq k}\bigg\{|\mathrm{row}_{i_l}(\Sigma^{-1/2})| \bigg\|\frac{\partial^{k-1} h(\mathbf{w})}{\prod_{\stackrel{1\leq j\leq k}{j\not=l}}\partial w_{i_j}}\bigg\|\Bigg\}, \quad k\geq 2,
\end{equation}
where $|\mathrm{row}_{i_l}(\Sigma^{-1/2})|$ is the Euclidean norm of the $i_l$-th row of $\Sigma^{-1/2}$.  Similar bounds for the derivatives of $f$ as a $k$-linear form can also be found in \cite{gaunt rate} and \cite{meckes}.  It was shown by \cite{daly} that the solution of (\ref{normal equation}) satisfies the bound
\begin{equation}\label{dalyb}\|f^{(k)}\|\leq 2\|h^{(k-2)}\|, \quad k\geq 3.
\end{equation}
This bound has the attractive property of involving two fewer derivatives of the test function $h$ than the solution $f$, although this improvement is not possible for multivariate case (see \cite{raic clt}). 

 When applying Stein's method to derive bounds for normal and multivariate normal approximation one typically requires bounds on at least the third order derivatives of the solution of the Stein equation.  In the univariate case the bound (\ref{dalyb}) with $k=3$ can be used to derive bounds in the Wasserstein distance, but in the multivariate case the bounds (\ref{cheque}) and (\ref{charm}) only allow for bounds to be given in the weaker $d_{\mathcal{H}_p}$ metric for $p\geq2$.  For this reason, until recently  bounds for multivariate normal approximation were mostly given in smooth Wasserstein metrics with $p\geq2$; bounds with sub-optimal order could be given in stronger metrics by applying smoothing techniques (see, for example, part 2 of Proposition 1.2 of \cite{ross}), though.  The very recent works of \cite{bon,cfp17,fang18,f18b,gms} have, however,  used novel implementations of Stein's method that bypass these technical difficulties and have established optimal or near-optimal bounds on the rate of convergence of the usual standardised sum of independent random vectors to the limiting multivariate normal distribution, under a variety of different assumptions.  We also refer the reader to \cite{zhai} for a recent complementary reference that attacks this problem without using Stein's method.  The approach taken in this paper is to obtain suitable analogues of the bounds (\ref{cheque}) -- (\ref{dalyb}) for unbounded test functions.  As such, for the multivariate case $d\geq2$ all quantitative limit theorems derived in this paper will be given in smooth Wasserstein metrics with $p\geq2$.


\subsection{Functions of multivariate normal random variables and a general transfer principle}\label{sec2.1}

By the continuous mapping theorem, if a sequence of $d$-dimensional random vectors  $(\mathbf{W}_n)_{n\geq1}$ converges in distribution to a multivariate normal random variable $\Sigma^{1/2}\mathbf{Z}$, then for any continuous function $g:\mathbb{R}^d\rightarrow\mathbb{R}$, $(g(\mathbf{W}_n))_{n\geq1}$ converges in distribution to $g(\Sigma^{1/2}\mathbf{Z})$. In this paper, we develop Stein's method for the problem of obtaining explicit bounds for rate of convergence of the sequence of random variables $(g(\mathbf{W}_n))_{n\geq1}$ to $g(\Sigma^{1/2}\mathbf{Z})$.  (From now, for ease of notation, we shall drop the subscript from $\mathbf{W}_n$.)  The general approach that shall be described in this section can in principle be applied to treat any prelimit of the form $g(\mathbf{W})$, where $g$ satisfies certain differentiability and growth rate conditions (which will be described shortly) and $\mathbf{W}$ can be well-approximated by a multivariate normal random variable by Stein's method for multivariate normal approximation.  However, the quantitative limit theorems that are derived in Section \ref{sec3rd} only treat the (important) case that the components of $\mathbf{W}$ are standardised sums of independent random variables; some reasons for imposing this restriction are given in Section \ref{secsum}.


Many standard probability distributions arise naturally as functions of multivariate normal random variables, such as the chi-square ($\chi_{(d)}^2\stackrel{\mathcal{D}}{=}Z_1^2+\cdots+ Z_d^2$), chi, log-normal and $t$-distribution; for further examples see \cite{stuart}.  Moreover, many widely used statistics arise as functions of asymptotically multivariate normally distributed random variables, such as Pearson's statistic, Friedman's statistic and the popular $D_2$, $D_2^S$ and $D_2^*$ statistics from alignment-free sequence comparison (see \cite{lippert} and \cite{waterman}).  Also, limiting distributions involving functions of multivariate normal random variables have recently occurred in the context of Malliavin calculus (the Malliavin-Stein method is described in detail in \cite{np12}), such as variance-gamma \cite{eichelsbacher} and linear combinations of chi-square random variables \cite{aaps17, azmooden}; see also \cite{eden2} for other non-normal limits.  

One of the strengths of Stein's method is that it is readily adapted to other distributions; for a comprehensive overview see \cite{ley}.  In particular, the method has been extended to many distributions that occur as functions of multivariate normal random variables, such as the chi-square \cite{gaunt chi square}, \cite{luk}, chi \cite{pekoz3}, half-normal \cite{dobler}, variance-gamma \cite{gaunt vg}, products  of normal and chi-square random variables \cite{gaunt pn, gaunt ngb, gms16} and linear combinations of centered chi-square random variables \cite{aaps19}.


In adapting Stein's method to these distributions, a suitable Stein equation for the distribution needs to be found, together with bounds on the solution and its (lower order) derivatives.
For certain distributions, despite recent advances (see \cite{dgv15}), this can be difficult; for example, consider the product normal distribution \cite{gaunt pn} for which only limited progress has been made towards obtaining bounds for the derivatives of the solution.  The approach described in this paper, which involves considering the multivariate normal Stein equation rather than the distribution's specific Stein equation removes this difficulty by treating distributions that arise as functions of multivariate normal random variables in a general framework (see also \cite{aaps17, nps2014} for recent works that derive approximation theorems without directly bounding the solution of the Stein equation for the limit law).    

Moreover, in certain situations it may be more natural to frame a problem in terms of the multivariate normal Stein equation than the Stein equation for the limiting distribution.  Recently, \cite{gaunt chi square} used such an approach to obtains bounds on the rate of convergence of Pearson's statistic to its limiting chi-square distribution.  


To obtain distributional approximations for statistics that are asymptotically distributed as a function of a multivariate normal, we simply consider the multivariate normal Stein equation (\ref{mvn145}) with test function $h(g(\cdot))$:
\begin{equation} \label{mvng} \nabla^T\Sigma\nabla f(\mathbf{w})-\mathbf{w}^T\nabla f(\mathbf{w})=h(g(\mathbf{w}))-\mathbb{E}h(g(\Sigma^{1/2}\mathbf{Z})).
\end{equation}  
One can then bound the expectation 
\begin{equation}\label{emvn}\mathbb{E}[\nabla^T\Sigma\nabla f(\mathbf{W})-\mathbf{W}^T\nabla f(\mathbf{W})]
\end{equation}
using the various coupling techniques developed for multivariate normal approximation (see \cite{goldstein 2, goldstein1, reinert 1, meckes}).  However, in general the derivatives of the test function $h(g(\mathbf{w}))$ will be unbounded (for the $\chi_{(1)}^2$ distribution, $g(w)=w^2$ and $g'(w)=2w$) and therefore the derivatives of the solution 
\begin{equation}\label{mvnsolnh}f(\mathbf{w})=-\int_{0}^{\infty}[\mathbb{E}h(g(\mathrm{e}^{-s}\mathbf{w}+\sqrt{1-\mathrm{e}^{-2s}}\Sigma^{1/2}\mathbf{Z}))-\mathbb{E}h(g(\Sigma^{1/2}\mathbf{Z}))]\,\mathrm{d}s
\end{equation}
will also in general be unbounded.  Therefore one cannot apply inequalities (\ref{cheque}), (\ref{charm}) and (\ref{dalyb}) to bound the solution's derivatives.  This simple, but powerful approach was first used by \cite{pickett} and \cite{reinert 0} in which the authors invoked the multivariate normal Stein equation to derive bounds for chi-square approximation.  

In this paper, we develop a general theory based on this approach.  The approach will be effective when the prelimit random variable is of the form $g(\mathrm{W})$, where $g:\mathbb{R}^d\rightarrow\mathbb{R}$ is sufficiently regular and $\mathrm{W}$ is well-approximated by a multivariate normal random variable and is such that the expectation (\ref{emvn}) can be estimated using one of the standard couplings for Stein's method for multivariate normal approximation.  There may be instances in which one cannot decompose the prelimit is such a way, in which case using the Stein equation for the limit distribution may prove to be more fruitful.  However, the combination of the fact there is a well established literature on such coupling techniques and that the estimates we obtain for the derivatives of the solution of the Stein equation with test function $h(g(\cdot))$ (see Section \ref{sec2nd}) are relatively simple means that there is potentially a wide range of problems that can be tackled via the techniques developed in this paper.  Indeed, as already noted it is natural to view the approximation of Pearson's statistic in this way, and other important statistics can also be treated; see \cite{gr16}.  The $D_2$ and $D_2^*$ statistics from alignment-free sequence comparison  also seem to naturally fall into this framework; see Section \ref{d2sec} for further details.

\subsection{Summary of results}\label{secsum}

In Section \ref{sec2nd}, we obtain general bounds, which apply for all $g$ that satisfy certain differentiability and growth rate conditions, for the derivatives of the solution (\ref{mvnsolnh}) of the $\mathrm{MVN}(\mathbf{0},\Sigma)$ Stein equation with test function $h(g(\cdot))$.  As special cases of these general bounds, we obtain bounds for the case that the lower order partial derivatives of $g$ have a polynomial ($A+B\sum_{k=1}^d|w_k|^{r_k}$, where the $r_k$ are non-negative) or exponential growth rate ($A\exp(t\sum_{k=1}^d|w_k|^c)$, where $0<c\leq 2$).  Our bounds for the derivatives of the solution of the Stein equation are in general unbounded as $|\mathbf{w}|\rightarrow\infty$, which means that more care is needed in bounding the quantity (\ref{emvn}) than in the usual multivariate normal setting in which the uniform bounds (\ref{cheque}), (\ref{charm}) and (\ref{dalyb}) can be applied. 

Since our bounds for the partial derivatives of the solution of the Stein equation involve derivatives of the test function $h$, the approximation theorems considered in this paper will only hold for smooth test functions.  For the univariate case $g:\mathbb{R}\rightarrow\mathbb{R}$ some of our bounds will be given in the Wasserstein metric, but all of our bounds for multivariate case $g:\mathbb{R}^d\rightarrow\mathbb{R}$, $d\geq2$, will be given in the smooth Wasserstein metric $d_{\mathcal{H}_p}$ for $p\geq2$ (see Remark \ref{wass}).  Bounds resulting from an application of Stein's method are often given in such smooth test function metrics, particularly in multivariate settings in which there are often technical difficulties in obtaining bounds in non-smooth metrics or when faster than $O(n^{-1/2})$ convergence rates are sought (see, for example, \cite{bj3, dobler beta,f18, gaunt chi square, goldstein, goldstein1, reinert 1}). Bounds for non-smooth test functions are often more informative (see, for example, \cite{gotze}), although, as noted by \cite{goldstein}, an advantage of working with smooth test functions is that it is sometimes possible to obtain improved error bounds that may not hold for non-smooth test functions.  This feature will be exploited in this paper: our proofs rely on the assumption that the test functions are smooth.

In Section \ref{sec3rd}, we illustrate our approach in the following setting. We consider the case that the derivatives of $g$ have polynomial growth and that the components of $\mathbf{W}=(W_1,\ldots,W_d)^T$ are given by $W_j=\frac{1}{\sqrt{n_j}}\sum_{i=1}^{n_j}X_{ij}$, where the $X_{ij}$ are independent random variables with mean zero and unit variance.  We impose the polynomial growth assumption to keep calculations manageable, but similar bounds could be obtained under the assumption of exponential growth rate; see Remark \ref{rem3.6} for further details.
It is also quite a mild assumption, as many important statistics satisfy such an assumption.  We study the case that the components of $\mathbf{W}$ are sums of independent random variables for several reasons.  Firstly, one can bound the expectation (\ref{emvn}) using local couplings, one of the simplest couplings for multivariate normal approximation via Stein's method.  This allows for a clear exposition of the transfer principle, in which the focus is on the techniques developed in this paper, rather than the intricacies of the coupling technique.  Secondly, it is possible to carry out a quite detailed investigation into the rate of convergence of $g(\mathbf{W})$ to $g(\mathbf{Z})$, which could provide valuable insights into more general settings.  Indeed, this is one of the main contributions of this paper.  

Suppose that the functions $h$ and $g$ are sufficiently regular.  Then, roughly speaking, our general bounds (Theorems \ref{winfirst1}--\ref{multievenguni}) can be summarised as follows.  If the first $p$ moments of the $X_{ij}$ are equal to those of the standard normal distribution, then our bound on the quantity of interest $|\mathbb{E}h(g(\mathbf{W}))-\mathbb{E}h(g(\mathbf{Z}))|$ is of order $n^{-(p-1)/2}$, where $n=\min_{1\leq j\leq d} n_j$.  Matching moments limit theorems with faster than $O(n^{-1/2})$ convergence rates have appeared in the context of Stein's method in the papers \cite{f18,gaunt rate, goldstein, lefevre}, and our result generalises the results of \cite{gaunt rate, goldstein}.
That matching moments may result in faster convergence rates is also know in other contexts; see, for example \cite{hall}.  Perhaps more interestingly, if $p$ is an even integer and we also suppose that $g$ is an even function ($g(-\mathbf{w})=g(\mathbf{w})$ for all $\mathbf{w}\in\mathbb{R}^d$), then we can use symmetry considerations, as introduced by \cite{gaunt chi square}, to improve this convergence rate further to order $n^{-p/2}$, a rate which cannot be improved (see Proposition \ref{prop3.5}).  In particular, $g(\mathbf{W})$ converges to $g(\mathbf{Z})$ at rate $n^{-1}$ even if the third moments of the $X_{ij}$ are non-zero.  This result generalises those of \cite{gaunt vg} and \cite{gaunt chi square} since their chi-square and variance-gamma statistics are of the form $g(\mathbf{W})$.  As far as this author is aware, the identification of the general condition that $g$ being an even function results in faster convergence rates in a smooth test function metric is an original contribution of this paper.   

By carrying out a detailed investigation of the case that the components of $\mathbf{W}$ are sums of independent random variables, we gain insight into more complex settings that one may encounter in applications.  Consider, for example, Pearson's statistic.  The statistic falls into the framework described above, with $g$ being an even function with derivatives of polynomial growth, with the sole exception that the assumption that the $X_{ij}$ (which would be the normalised indicator random variable that trial $i$ falls into class $j$) are independent fails.  Given the insight from Theorems \ref{winfirst1} -- \ref{multievenguni} it is, however, awfully tempting to speculate that if the dependence is `sufficiently weak' then a $O(n^{-1})$ bound can be derived.  This is indeed the case \cite{gaunt chi square}, and further general results are given in \cite{gr16}.  

We end in Section \ref{secex} with several examples that are chosen to illuminate the theory developed in this paper.  We consider normal  ($g(w)=w$), chi-square ($g(w)=w^2$) and Gaussian polynomial ($g(w)=H_n(w)$, where $H_n$ is the $n$-th Hermite polynomial) approximation to illustrate the general bounds of Theorems \ref{winfirst1}--\ref{multievenguni} in a concrete setting and to provide a comparison with existing bounds from the Stein's method literature.  We also demonstrate, through binomial and Poisson examples, how one can obtain approximations for expectations of smooth (unbounded) functions of random variables by corresponding expectations of normal random variables.  Tighter bounds can be obtained when the smooth function is even.  In Section \ref{deltasec}, we consider a more involved application to the rate of convergence in the delta method.  Finally, we consider an application to sequence comparison.  In Section \ref{d2sec}, we outline how an extension of the theory developed in this paper could be used to obtain quantitative limit theorems for the $D_2^*$ statistic from alignment-free sequence comparison, and in Section \ref{binsec} we apply Theorems \ref{winfirst1} and \ref{multieveng} to derive bounds to quantify the variance-gamma approximation of the statistic in the special case of binary sequence comparison.

\section{Bounds for derivatives of the solutions of the normal and multivariate normal Stein equations}\label{sec2nd}

\subsection{Preliminary results}

We begin this section by obtaining a simple bound for the partial derivatives of the test function $h(g(\cdot))$.  Before deriving this bound, we state some preliminary results.  The first is a multivariate generalisation of the Fa\`{a} di Bruno formula for $n$-th order derivatives of composite functions, due to \cite{ma}:
\begin{equation*}\frac{\partial^n}{\prod_{j=1}^n\partial w_{i_j}}h(g(\mathbf{w}))=\sum_{\pi\in\Pi}h^{(|\pi|)}(g(\mathbf{w}))\cdot\prod_{B\in\pi}\frac{\partial^{|B|}g(\mathbf{w})}{\prod_{j\in B}\partial w_{i_j}},
\end{equation*}
where $\pi$ runs through the set $\Pi$ of all partitions of the set $\{1,\ldots,n\}$, the product is over all of the parts $B$ of the partition $\pi$, and $|S|$ is the cardinality of the set $S$.  It is useful to note that the number of partitions of $\{1,\ldots,n\}$ into $k$ non-empty subsets is given by the Stirling number of the second kind ${n\brace k}=\frac{1}{k!}\sum_{j=0}^k(-1)^{k-j}\binom{k}{j}j^n$ (see \cite{olver}).

We now introduce two classes of functions that will be used throughout this paper.  We say that the function $h:I\subseteq\mathbb{R}\rightarrow\mathbb{R}$ belongs to the class $C_b^n(I)$ if $h^{(n-1)}$ exists and is absolutely continuous, with derivatives up to $n$-th order bounded. For a given $P$, we say that the function $g:\mathbb{R}^d\rightarrow\mathbb{R}$ belongs to the class $C_{P}^n(\mathbb{R}^d)$ if all $n$-th order partial derivatives of $g$ exist and are such that, for all $\mathbf{w}\in\mathbb{R}^d$, 
\begin{equation*}\bigg|\frac{\partial^k}{\prod_{j=1}^k\partial w_{i_j}}g(\mathbf{w})\bigg|^{n/k}\leq P(\mathbf{w}), \quad k=1,\ldots,n.
\end{equation*}
If $g\in C_{P}^n(\mathbb{R}^d)$ then it is easy to see that, for all $\mathbf{w}\in\mathbb{R}^d$,
\begin{equation*}\bigg|\prod_{B\in\pi}\frac{\partial^{|B|}g(\mathbf{w})}{\prod_{j\in B}\partial w_{i_j}}\bigg|\leq P(\mathbf{w}).
\end{equation*}
This inequality allows us to obtain a compact bound for the partial derivatives of the test function $h(g(\cdot))$, which in turn allows us to obtain relatively simple bounds for the solution (\ref{mvnsolnh}) of the Stein equation (\ref{mvng}).

\begin{lemma}\label{bell lem}Suppose $h\in C_b^n(\mathbb{R})$ and $g\in C_P^n(\mathbb{R}^d)$, where $n\geq1$.  Then, for all $\mathbf{w}\in\mathbb{R}^d$,
\begin{equation}\label{hgdiff}\bigg|\frac{\partial^n}{\prod_{j=1}^n\partial w_{i_j}}h(g(\mathbf{w}))\bigg|\leq h_nP(\mathbf{w}),
\end{equation}
where $h_n=\sum_{k=1}^n{n\brace k}\|h^{(k)}\|$.
\end{lemma}

\begin{proof}From the above it is clear that
\begin{equation*}\bigg|\frac{\partial^n}{\prod_{j=1}^n\partial w_{i_j}}h(g(\mathbf{w}))\bigg|\leq \sum_{\pi\in \Pi}\|h^{(|\pi|)}\|\cdot P(\mathbf{w})=\sum_{k=1}^n{n\brace k}\|h^{(k)}\|\cdot P(\mathbf{w}), 
\end{equation*}
as required.
\end{proof}

We will also make use of the following lemma.

\begin{lemma}\label{constlem}Suppose $h\in C_b^n(\mathbb{R})$ and $g\in C_{\alpha}^n(\mathbb{R})$, where $n\geq2$ and $\alpha\geq0$ is a constant.  Then the solution (\ref{mvnsolnh}) of the Stein equation (\ref{mvng}) is bounded by
\begin{equation*}\|wf^{(n)}(w)\|\leq \|(h\circ g)^{n-1}\|\leq \alpha h_{n-1}.
\end{equation*}
\end{lemma}

\begin{proof}The inequality $\|wf^{(n)}(w)\|\leq \|(h\circ g)^{n-1}\|$ is given in Lemma 2.5 of \cite{gaunt rate} and the second inequality follows from Lemma \ref{bell lem}. 
\end{proof}

\subsection{General bounds for the solution}\label{sec2point2}

Here, we obtain some general bounds for the solution of the multivariate normal Stein equation with test function $h(g(\cdot))$.  We begin with the following lemma, the proof of which is similar to that of Proposition 2.1 of \cite{gaunt rate}.

\begin{lemma}\label{Plem}Suppose $\Sigma$ is non-negative definite and that $h\in C_b^n(\mathbb{R})$ and $g\in C_P^n(\mathbb{R}^d)$, where $n\geq1$.  Then, for all $\mathbf{w}\in\mathbb{R}^d$, the solution (\ref{mvnsoln}) of the Stein equation (\ref{mvn145}) satisfies the bound
\begin{align}\label{ensure1}\bigg|\frac{\partial^nf(\mathbf{w})}{\prod_{j=1}^n\partial w_{i_j}}\bigg|&\leq h_n\int_0^{\infty}\mathrm{e}^{-ns}\mathbb{E}P(\mathbf{z}_{s,\mathbf{w}}^{\Sigma^{1/2}\mathbf{Z}})\,\mathrm{d}s,
\end{align}
where
\begin{equation*}\mathbf{z}_{s,\mathbf{w}}^{\mathbf{x}}=\mathrm{e}^{-s}\mathbf{w}+\sqrt{1-\mathrm{e}^{-2s}}\mathbf{x},
\end{equation*}
provided the integral exists.  

Suppose now that $\Sigma$ is positive definite and that $h\in C_b^{n-1}(\mathbb{R})$ and $g\in C_P^{n-1}(\mathbb{R}^d)$, where $n\geq2$.  Then, provided the integral exists, we have, for all $\mathbf{w}\in\mathbb{R}^d$,
\begin{align}\label{ensure2}\bigg|\frac{\partial^{n}f(\mathbf{w})}{\prod_{j=1}^{n}\partial w_{i_j}}\bigg|&\leq h_{n-1}\min_{1\leq l\leq d}\int_0^{\infty}\frac{\mathrm{e}^{-ns}}{\sqrt{1-\mathrm{e}^{-2s}}}\mathbb{E}\big|(\Sigma^{-1/2}\mathbf{Z})_lP(\mathbf{z}_{s,\mathbf{w}}^{\Sigma^{1/2}\mathbf{Z}})\big|\,\mathrm{d}s.
\end{align}
\end{lemma}

\begin{proof}Firstly, by the dominated convergence theorem, we obtain the following  expression for the $n$-th order partial derivatives of the solution (\ref{mvnsoln}):
\begin{equation}\label{int1}\frac{\partial^nf(\mathbf{w})}{\prod_{j=1}^n\partial w_{i_j}}=-\int_0^{\infty}\mathrm{e}^{-ns}\mathbb{E}\bigg[\frac{\partial^n }{\prod_{j=1}^n\partial w_{i_j}}h(g(\mathbf{z}_{s,\mathbf{w}}^{\Sigma^{1/2}\mathbf{Z}}))\bigg]\,\mathrm{d}s.
\end{equation}
By Lemma \ref{bell lem} and the assumptions that $h\in C_b^n(\mathbb{R})$ and $g\in C_P^n(\mathbb{R}^d)$, we have that, for all $\mathbf{w}\in\mathbb{R}^d$,
\begin{equation}\label{dommm1}\bigg|\frac{\partial^n}{\prod_{j=1}^n\partial w_{i_j}}h(g(\mathbf{w}))\bigg|\leq h_nP(\mathbf{w}).
\end{equation}
Combining this bound with (\ref{int1}) then yields (\ref{ensure1}).  Lastly, we note that we were able to apply the dominated convergence theorem to interchange the operations of integration and differentiation in virtue of the dominating function in (\ref{dommm1}) and the assumption that the integral in (\ref{ensure1}) exists.  We shall also apply the dominated convergence theorem to interchange the operations of integration and differentiation later in the proof, and similar justifications can be given.

Now, we prove inequality (\ref{ensure2}).  Suppose $\Sigma$ is positive definite, so that $\Sigma^{-1/2}$ exists.  We proceed by using a very similar calculation to the one used in the proof of Proposition 2.1 of \cite{gaunt rate} to find another expression for the $n$-th order partial derivatives of the solution (\ref{mvnsoln}).  We begin by writing the solution (\ref{mvnsoln}) in the form
\begin{equation*}f(\mathbf{w})=-\int_{0}^{\infty}\!\int_{\mathbb{R}^d}[h(g(\mathrm{e}^{-s}\mathbf{w}+\sqrt{1-\mathrm{e}^{-2s}}\mathbf{x}))-\mathbb{E}h(g(\Sigma^{1/2}\mathbf{Z}))]p(\mathbf{x})\,\mathrm{d}\mathrm{x}\,\mathrm{d}s,
\end{equation*}
where $p(\mathbf{x})=(2\pi)^{-d/2}(\det(\Sigma))^{-1/2}\exp\big(-\frac{1}{2}\mathbf{x}^T\Sigma^{-1}\mathbf{x}\big)$ is the $\mathrm{MVN}(\mathbf{0},\Sigma)$ density. 
Making the change of variable $\mathbf{y}=\mathrm{e}^{-s}\mathbf{w}+\sqrt{1-\mathrm{e}^{-2s}}\mathbf{x}$ gives
\begin{equation*}
f(\mathbf{w})=-\int_0^{\infty}\!\int_{\mathbb{R}^d}\frac{1}{(1-\mathrm{e}^{-2s})^{d/2}}[h(g(\mathbf{y}))-\mathbb{E}h(g(\Sigma^{1/2}\mathbf{Z}))]p\bigg(\frac{\mathbf{y}-\mathrm{e}^{-s}\mathbf{w}}{\sqrt{1-\mathrm{e}^{-2s}}}\bigg)\,\mathrm{d}\mathbf{y}\,\mathrm{d}s.
\end{equation*}
An application of the dominated convergence theorem gives that
\begin{align}\frac{\partial f(\mathbf{w})}{\partial w_{i}}&=-\int_0^{\infty}\!\int_{\mathbb{R}^d}\frac{\mathrm{e}^{-s}}{(1-\mathrm{e}^{-2s})^{(d+1)/2}}(\Sigma^{-1}(\mathrm{y}-\mathrm{e}^{-s}\mathrm{w}))_i[h(g(\mathbf{y}))-\mathbb{E}h(g(\Sigma^{1/2}\mathbf{Z}))]\nonumber\\
&\quad\times p\bigg(\frac{\mathbf{y}-\mathrm{e}^{-s}\mathbf{w}}{\sqrt{1-\mathrm{e}^{-2s}}}\bigg)\,\mathrm{d}\mathbf{y}\,\mathrm{d}s\nonumber\\
&=-\int_0^{\infty}\!\int_{\mathbb{R}^d}\frac{\mathrm{e}^{-s}}{\sqrt{1-\mathrm{e}^{-2s}}}(\Sigma^{-1}\mathbf{x})_i[h(g(\mathrm{e}^{-s}\mathbf{w}+\sqrt{1-\mathrm{e}^{-2s}}\mathbf{x}))-\mathbb{E}h(g(\Sigma^{1/2}\mathbf{Z}))] p(\mathbf{x})\,\mathrm{d}\mathbf{x}\,\mathrm{d}s\nonumber\\
&=-\int_0^{\infty}\frac{\mathrm{e}^{-s}}{\sqrt{1-\mathrm{e}^{-2s}}}\mathbb{E}\Big[(\Sigma^{-1/2}\mathbf{Z})_i[h(g(\mathrm{e}^{-s}\mathbf{w}+\sqrt{1-\mathrm{e}^{-2s}}\Sigma^{1/2}\mathbf{Z}))-\mathbb{E}h(g(\Sigma^{1/2}\mathbf{Z}))]\Big]\,\mathrm{d}s,
\end{align}
where we used the formula $\frac{\partial}{\partial x_{i}}(\mathbf{x}^T\Sigma^{-1}\mathbf{x})=2(\Sigma^{-1}\mathbf{x})_i$ in obtaining the first equality. By another application of the dominated convergence we have, for any $l\in\{1,\ldots,n\}$, 
\begin{equation}\label{ff6ffvi}\frac{\partial^n f(\mathbf{w})}{\prod_{j=1}^n\partial w_{i_j}}=-\int_{0}^{\infty}\frac{\mathrm{e}^{-ns}}{\sqrt{1-\mathrm{e}^{-2s}}}\mathbb{E}\bigg[(\Sigma^{-1/2}\mathbf{Z})_{i_l}\frac{\partial^{n-1} }{\prod_{\stackrel{1\leq j\leq n-1}{j\not=l}}\partial w_{i_j}}h(g(\mathrm{e}^{-s}\mathbf{w}+\sqrt{1-\mathrm{e}^{-2s}}\Sigma^{1/2}\mathbf{Z}))\bigg]\,\mathrm{d}s.
\end{equation}
Finally, we can apply Lemma \ref{bell lem} to obtain (\ref{ensure2}), which completes the proof.
\end{proof}

So far, we have imposed no restrictions on the dominating function $P$ other than it is non-negative and that the integrals of Lemma \ref{Plem} exist.  We now introduce some conditions, which ensure that the integrals of Lemma \ref{Plem} exist and can be bounded relatively easily. As we will see in Examples \ref{ex1} and \ref{ex2} below, these conditions are not restrictive and allow many classes of functions to be considered.

\begin{assumption}\label{ass1}
We suppose that $P$ can be written as $P(\mathbf{w})=\alpha+P_1(\mathbf{w})+P_2(\mathbf{w})$, where $\alpha$ is a non-negative constant and

\begin{enumerate}
\item[(i)] $P_1$ and $P_2$ are non-negative, non-decreasing functions, in the sense that, for any $\mathbf{w}\in\mathbb{R}^d$ and $a>1$, one has $P_i(\mathbf{w})\leq P_i(a\mathbf{w})$, $i=1,2$;

\item[(ii)] There exist non-negative constants $\beta_k$, $\gamma_k$ and $\delta_k$ such that, for any $\mathbf{w}_1,\ldots,\mathbf{w}_k\in\mathbb{R}^d$,
\begin{equation*}P_1(\mathbf{w}_1+\cdots+\mathbf{w}_k)\!\leq\!\beta_k\!\sum_{j=1}^k\!P_1(\mathbf{w}_k) \:\; \mbox{and} \:\; P_2(\mathbf{w}_1+\cdots+\mathbf{w}_k)\!\leq\!\gamma_k\!\prod_{j=1}^d\!P_2(\delta_k\mathbf{w}_k);
\end{equation*}

\item[(iii)] The expectations $\mathbb{E}P_1(\Sigma^{1/2}\mathbf{Z})$ and $\mathbb{E}P_2(\Sigma^{1/2}\mathbf{Z})$ exist; 

\item[(iii)'] The expectations $\mathbb{E}|(\mathbf{Z})_iP_1(\Sigma^{1/2}\mathbf{Z})|$ and $\mathbb{E}|(\mathbf{Z})_iP_2(\Sigma^{1/2}\mathbf{Z})|$ exist for all $i=1,\ldots,d$. 
\end{enumerate}

If $P$ satisfies (i)-(iii), we write $P\in \mathcal{F}$; if $P$ satisfies (i)-(iii)', we write $P\in \mathcal{F}_*$.
\end{assumption}



\begin{example}\label{ex1}\emph{Polynomial $P$}.  The function $P(\mathbf{w})=A+B\sum_{i=1}^d|w_i|^{r_i}$ clearly satisfies conditions (i) and (iii)'.  Condition (ii) can be verified by applying the crude inequality $|x_1+\cdots +x_k|^b\leq k^b(|x_1|^b+\cdots+|x_k|^b)$, $b\geq 0$.  Writing $P(\mathbf{w})=A+P_1(\mathbf{w})$, where $P_1(\mathbf{w})=B\sum_{i=1}^d|w_i|^{r_i}$, and using this inequality we deduce that
\begin{equation}\label{ex1eqn}P_1(\mathbf{w}_1+\cdots+\mathbf{w}_k)\leq k^r\sum_{j=1}^kP_1(\mathbf{w}_j),
\end{equation}
where $r=\max_{1\leq i\leq d}r_i$.  Hence, (ii) is satisfied with $\alpha=A$, $\beta_k=k^r$, $\gamma_k=0$ and $P_1(\mathbf{w})=B\sum_{i=1}^d|w_i|^{r_i}$.
\end{example}

\begin{example}\label{ex2}\emph{Exponential $P$}.  It is clear that $P(\mathbf{w})=A\exp(a\sum_{i=1}^d|w_i|^b)$, where $a,b>0$, satisfies (i).  To verify (ii), we use the inequality $|x_1+\cdots +x_k|^r\leq c_{k,r}(|x_1|^r+\cdots+|x_k|^r)$, where $c_{k,r}=\max\{1,k^{r-1}\}$, which improves on the crude inequality used above.  Let $(\mathbf{w}_j)_k=w_{jk}$.  Then
\begin{align*}P(\mathbf{w}_1+\cdots+\mathbf{w}_k)&=A\exp\bigg(a\sum_{i=1}^d|w_{1i}+\cdots+w_{ki}|^b\bigg) \leq A\exp\bigg(a\sum_{i=1}^dc_{k,b}\sum_{j=1}^k|w_{ji}|^b\bigg) \\
&=A\prod_{j=1}^k\exp\bigg(ac_{k,b}\sum_{i=1}^d|w_{ji}|^b\bigg)=A^{1-k}\prod_{j=1}^kP(c_{k,b}\mathbf{w}_j),
\end{align*}
and so (ii) holds with $\alpha=0$, $\beta_k=0$, $\gamma_k=A^{1-k}$, $\delta_k=c_{k,b}$ and $P_2(\mathbf{w})=A\exp(a\sum_{i=1}^d|w_i|^b)$.  Clearly, (iii)' holds if $b<2$ and $a>0$, or if $b=2$ and $0<a<1/2$.
\end{example}

\begin{proposition}\label{btlet}  Suppose $\Sigma$ is non-negative definite and that $h\in C_b^n(\mathbb{R})$ and $g\in C_P^n(\mathbb{R}^d)$, where $n\geq 1$ and $P\in\mathcal{F}$.  Then, for all $\mathbf{w}\in\mathbb{R}^d$, 
\begin{align}\label{btlet1}\bigg|\frac{\partial^nf(\mathbf{w})}{\prod_{j=1}^n\partial w_{i_j}}\bigg|&\leq\frac{h_n}{n}\Big[\alpha\!+\!\beta_2\big(\mathbb{E}P_1(\Sigma^{1/2}\mathbf{Z})\!+\!P_1(\mathbf{w})\big)
\!+\!\gamma_2P_2(\delta_2\mathbf{w})\mathbb{E}P_2(\delta_2\Sigma^{1/2}\mathbf{Z})\Big].
\end{align}
Suppose now that $\Sigma$ is positive definite and that $h\in C_b^{n-1}(\mathbb{R})$ and $g\in C_P^{n-1}(\mathbb{R}^d)$, where $n\geq 2$ and $P\in\mathcal{F}_*$.  Then, for all $\mathbf{w}\in\mathbb{R}^d$,
\begin{align}
\bigg|\frac{\partial^{n}f(\mathbf{w})}{\prod_{j=1}^{n}\partial w_{i_j}}\bigg|&\leq h_{n-1}\min_{1\leq l\leq d}
\Big[\alpha\mathbb{E}|(\Sigma^{-1/2}\mathbf{Z})_l|+\beta_2\big(\mathbb{E}|(\Sigma^{-1/2}\mathbf{Z})_lP_1(\Sigma^{1/2}\mathbf{Z})|\nonumber\\
\label{btlet2}&\quad+\!\mathbb{E}|(\Sigma^{1/2}\mathbf{Z})_l|P_1(\mathbf{w})\big)\!+\!\gamma_2\mathbb{E}|(\Sigma^{-1/2}\mathbf{Z})_lP_2(\delta_2\Sigma^{1/2}\mathbf{Z})|P_2(\delta_2\mathbf{w})\Big].
\end{align}
\end{proposition}

\begin{proof}Since $P\in\mathcal{F}$, 
\begin{align}P(\mathbf{z}_{s,\mathbf{w}}^{\Sigma^{1/2}\mathbf{Z}})&= \alpha+P_1(\mathbf{z}_{s,\mathbf{w}}^{\Sigma^{1/2}\mathbf{Z}})+P_2(\mathbf{z}_{s,\mathbf{w}}^{\Sigma^{1/2}\mathbf{Z}})\nonumber \\
&\leq\alpha+\beta_2\big(P_1(\mathrm{e}^{-s}\mathbf{w})+P_1(\sqrt{1-\mathrm{e}^{-2s}}\Sigma^{1/2}\mathbf{Z})\big)+\gamma_2P_2(\delta_2\mathrm{e}^{-s}\mathbf{w})P_2(\delta_2\sqrt{1-\mathrm{e}^{-2s}}\Sigma^{1/2}\mathbf{Z}) \nonumber \\
\label{pinequality}&\leq\alpha+\beta_2\big(P_1(\mathbf{w})+P_1(\Sigma^{1/2}\mathbf{Z})\big)+\gamma_2P_2(\delta_2\mathbf{w})P_2(\delta_2\Sigma^{1/2}\mathbf{Z}),
\end{align}
where the final inequality follows from property (i).  Substituting (\ref{pinequality}) into the integral inequalities (\ref{ensure1}) and (\ref{ensure2}), respectively, gives the bounds
\[\bigg|\frac{\partial^{n}f(\mathbf{w})}{\prod_{j=1}^{n}\partial w_{i_j}}\bigg|\leq h_{n}\int_0^{\infty}\mathrm{e}^{-ns}\Big[\alpha\!+\!\beta_2\big(\mathbb{E}P_1(\Sigma^{1/2}\mathbf{Z})\!+\!P_1(\mathbf{w})\big)
\!+\!\gamma_2P_2(\delta_2\mathbf{w})\mathbb{E}P_2(\delta_2\Sigma^{1/2}\mathbf{Z})\Big]\,\mathrm{d}s,\]
and
\begin{align*}\bigg|\frac{\partial^{n}f(\mathbf{w})}{\prod_{j=1}^{n}\partial w_{i_j}}\bigg|&\leq h_{n-1}\min_{1\leq l\leq d}\int_0^{\infty}\frac{\mathrm{e}^{-ns}}{\sqrt{1-\mathrm{e}^{-2s}}}\mathbb{E}\big|(\Sigma^{-1/2}\mathbf{Z})_l(\alpha+\beta_2\big(P_1(\mathbf{w})+P_1(\Sigma^{1/2}\mathbf{Z})\big)\\
&\quad+\gamma_2P_2(\delta_2\mathbf{w})P_2(\delta_2\Sigma^{1/2}\mathbf{Z})))\big|\,\mathrm{d}s,
\end{align*}
whence on evaluating the integral $\int_0^\infty\mathrm{e}^{-ns}\,\mathrm{d}s=\frac{1}{n}$ and using the bound $\int_0^{\infty}\frac{\mathrm{e}^{-ns}}{\sqrt{1-\mathrm{e}^{-2s}}}\,\mathrm{d}s\leq \int_0^{\infty}\frac{\mathrm{e}^{-2s}}{\sqrt{1-\mathrm{e}^{-2s}}}\,\mathrm{d}s=1$, $n\geq 2$, we obtain inequalities (\ref{btlet1}) and (\ref{btlet2}), respectively.
\end{proof}

\begin{corollary}Fix $d=1$, $\Sigma=1$ and let $n\geq 2$. Suppose $h\in C_b^{n-1}(\mathbb{R})$ and $g\in C_P^{n-1}(\mathbb{R})$, where $P\in\mathcal{F}_*$.  Then, for all $w\in\mathbb{R}$, 
\begin{align}\label{wfirst}|wf^{(n)}(w)|&\leq h_{n-1}\Big[\alpha+\beta_2|w|\big(\mathbb{E}|ZP_1(Z)|+P_1(w)\big)+\gamma_2|w|P_2(\delta_2w)\mathbb{E}P_2(\delta_2Z)\Big].
\end{align}
\end{corollary}

\begin{proof}Due to the decomposition $P(w)=\alpha+P_1(w)+P_2(w)$, we can write the solution as $f=f_{g_1}+f_{g_2}$, where $g_1\in C_{\alpha}^{n-1}(\mathbb{R})$ and $g_2\in C_{P_1+P_2}^{n-1}(\mathbb{R})$.  By the triangle inequality, $|f^{(n)}(w)|\leq|f^{(n)}_{g_1}(w)|+|f^{(n)}_{g_2}(w)|$, and bounding these two quantities using Lemma \ref{constlem} and inequality (\ref{btlet2}) of Proposition \ref{btlet}, respectively, and that $\mathbb{E}|Z|<1$, leads to the desired bound.  
\end{proof}

In the univariate case, we can obtain a bound which involves two fewer derivatives of $h$ and $g$ than of the solution $f$ (compare the following Proposition to inequality (\ref{dalyb})); an improvement that is not possible in the multivariate case (see \cite{raic clt}). 

\begin{proposition}Fix $d=1$, $\Sigma=1$ and let $n\geq 3$.  Suppose $h\in C_b^{n-2}(\mathbb{R})$ and $g\in C_P^{n-2}(\mathbb{R})$, where $P\in\mathcal{F}_*$.  Then, for all $w\in\mathbb{R}$, 
\begin{align}|f^{(n)}(w)|&\leq h_{n-2}\Big[3\alpha+P_1(w)+P_2(w)+\beta_2\big(\mathbb{E}P_1(Z)+|w|\mathbb{E}|ZP_1(Z)|\nonumber\\
\label{f2soln}&\quad+(1+|w|)P_1(w)\big)+\gamma_2\big(\mathbb{E}P_2(\delta_2Z)P_2(\delta_2w)+\mathbb{E}|ZP_2(\delta_2 Z)||wP_2(\delta_2w)|\big)\Big].
\end{align}
\end{proposition}

\begin{proof}The standard normal Stein equation is $f''(w)-wf'(w)=h(g(w))-\mathbb{E}h(g(Z))$.  By a straightforward induction on $n$,
\[f^{(n)}(w)=wf^{(n-1)}(w)+(n-2)f^{(n-2)}(w)+(h\circ g)^{(n-2)}(w),\]
and applying the triangle inequality gives that, for every $w\in\mathbb{R}$,
\begin{equation*}|f^{(n)}(w)|\leq|wf^{(n-1)}(w)|+(n-2)|f^{(n-2)}(w)|+|(h\circ g)^{(n-2)}(w)|.
\end{equation*}
Bounding these terms using (\ref{wfirst}), (\ref{btlet1}) and (\ref{hgdiff}) yields (\ref{f2soln}).
\end{proof}

Now, we obtain estimates for the solution $\psi_m$ of the Stein equation
\begin{equation} \label{mvnpsi} \nabla^T\Sigma\nabla \psi_m(\mathbf{w})-\mathbf{w}^T\nabla \psi_m(\mathbf{w})=\frac{\partial^mf(\mathbf{w})}{\prod_{j=1}^m\partial w_{i_j}},
\end{equation}
where $f$ is the solution (\ref{mvnsolnh}) of the multivariate normal Stein equation (\ref{mvng}).  The Stein equation (\ref{mvnpsi}) plays an important role in Section \ref{sec3rd}.  We proceed as before and the following lemma is analogous to Lemma \ref{Plem}.

\begin{lemma}\label{lempppss}Suppose $\Sigma$ is non-negative definite and that $h\in C_b^{m+n}(\mathbb{R})$ and $g\in C_P^{m+n}(\mathbb{R}^d)$, $m,n\geq 1$.  Then, for all $\mathbf{w}\in\mathbb{R}^d$, the solution of the Stein equation (\ref{mvnpsi}) satisfies the bound
\begin{align}\bigg|\frac{\partial^n\psi_m(\mathbf{w})}{\prod_{j=1}^n\partial w_{i_j}}\bigg|&\leq h_{m+n}\int_0^{\infty}\!\int_0^{\infty}\mathrm{e}^{-(m+n)s}\mathrm{e}^{-nt}\mathbb{E}P(\mathbf{z}_{s,t,\mathbf{w}}^{\Sigma^{1/2}\mathbf{Z},\Sigma^{1/2}\mathbf{Z}'})\,\mathrm{d}s\,\mathrm{d}t, \nonumber
\end{align}
where
\begin{equation*}\mathbf{z}_{s,t,\mathbf{w}}^{\mathbf{x},\mathbf{y}}=\mathrm{e}^{-s-t}\mathbf{w}+\mathrm{e}^{-s}\sqrt{1-\mathrm{e}^{-2t}}\mathbf{y}+\sqrt{1-\mathrm{e}^{-2s}}\mathbf{x},
\end{equation*}
provided the integral exists.  Here $\mathbf{Z}'$ is an independent copy of $\mathbf{Z}$.  

Suppose now that $\Sigma$ is positive definite and that $h\in C_b^{m+n-2}(\mathbb{R})$ and $g\in C_P^{m+n-2}(\mathbb{R}^d)$, where $m,n\geq 1$ and $m+n\geq 3$.  Then, provided the integral exists, we have, for all $\mathbf{w}\in\mathbb{R}^d$,
\begin{align}\bigg|\frac{\partial^{n}\psi_{m}(\mathbf{w})}{\prod_{j=1}^{n}\partial w_{i_j}}\bigg|&\leq h_{m+n-2}\min_{1\leq k,l\leq d}\int_0^{\infty}\!\int_0^{\infty}\frac{\mathrm{e}^{-(m+n)s}}{\sqrt{1-\mathrm{e}^{-2s}}}\frac{\mathrm{e}^{-nt}}{\sqrt{1-\mathrm{e}^{-2t}}}\nonumber\\
&\quad\times\mathbb{E}\big|(\Sigma^{-1/2}\mathbf{Z})_k(\Sigma^{-1/2}\mathbf{Z}')_lP(\mathbf{z}_{s,t,\mathbf{w}}^{\Sigma^{1/2}\mathbf{Z},\Sigma^{1/2}\mathbf{Z}'})\big|\,\mathrm{d}s\,\mathrm{d}t.\nonumber
\end{align}
\end{lemma}

\begin{proof}The solution of (\ref{mvnpsi}) can be written as
\begin{equation*}\psi_m(\mathbf{w})=-\int_0^{\infty}\!\int_{\mathbb{R}^d}\frac{\partial^mf}{\prod_{j=1}^m\partial w_{i_j}}(\mathbf{z}_{t,\mathbf{w}}^{\mathbf{y}})p(\mathbf{y})\,\mathrm{d}\mathbf{y}\,\mathrm{d}t,
\end{equation*}
where $p$ is the probability density function of the random variable $\Sigma^{1/2}\mathbf{Z}$.  By the dominated convergence theorem,
\begin{align*}\frac{\partial^mf}{\prod_{j=1}^m\partial w_{i_j}}(\mathbf{z}_{t,\mathbf{w}}^{\mathbf{y}})
 &=-\int_0^{\infty}\!\int_{\mathbb{R}^d}\mathrm{e}^{-ms}\frac{\partial^m(h\circ g)}{\prod_{j=1}^m\partial w_{i_j}}(\mathrm{e}^{-s}\mathbf{z}_{t,\mathbf{w}}^{\mathbf{y}}+\sqrt{1-\mathrm{e}^{-2s}}\mathbf{x}) p(\mathbf{x})\,\mathrm{d}\mathbf{x}\,\mathrm{d}s \\
&=-\int_0^{\infty}\!\int_{\mathbb{R}^d}\mathrm{e}^{-ms}\frac{\partial^m(h\circ g)}{\prod_{j=1}^m\partial w_{i_j}}(\mathbf{z}_{s,t,\mathbf{w}}^{\mathbf{x},\mathbf{y}})p(\mathbf{x})\,\mathrm{d}\mathbf{x}\,\mathrm{d}s,
\end{align*}
and we can therefore write
\begin{align*}\psi_m(\mathbf{w})&=\int_0^{\infty}\!\int_0^{\infty}\!\int_{\mathbb{R}^{2d}}\mathrm{e}^{-ms}\frac{\partial^m(h\circ g)}{\prod_{j=1}^m\partial w_{i_j}}(\mathbf{z}_{s,t,\mathbf{w}}^{\mathbf{x},\mathbf{y}})p(\mathbf{x})p(\mathbf{y})\,\mathrm{d}\mathbf{x}\,\mathrm{d}\mathbf{y}\,\mathrm{d}t\,\mathrm{d}s.
\end{align*}
By again applying the dominating convergence theorem, we have
\begin{align}\frac{\partial^n\psi_m(\mathbf{w})}{\prod_{j=1}^n\partial w_{i_j}}&=\int_0^{\infty}\!\int_0^{\infty}\!\int_{\mathbb{R}^{2d}}\mathrm{e}^{-(m+n)s}\mathrm{e}^{-nt}\frac{\partial^{m+n}(h\circ g)}{\prod_{j=1}^{m+n}\partial w_{i_j}}(\mathbf{z}_{s,t,\mathbf{w}}^{\mathbf{x},\mathbf{y}}) p(\mathbf{x})p(\mathbf{y})\,\mathrm{d}\mathbf{x}\,\mathrm{d}\mathbf{y}\,\mathrm{d}t\,\mathrm{d}s\nonumber\\
\label{int13}&=\int_0^{\infty}\!\int_0^{\infty}\mathrm{e}^{-(m+n)s}\mathrm{e}^{-nt}\mathbb{E}\bigg[\frac{\partial^{m+n}(h\circ g)}{\prod_{j=1}^{m+n}\partial w_{i_j}}(\mathbf{z}_{s,t,\mathbf{w}}^{\Sigma^{1/2}\mathbf{Z},\Sigma^{1/2}\mathbf{Z}'})\bigg]\,\mathrm{d}t\,\mathrm{d}s,
\end{align}
which, on applying integration by parts twice, can be rewritten as
\begin{align}\frac{\partial^{n}\psi_{m}(\mathbf{w})}{\prod_{j=1}^{n}\partial w_{i_j}}
&=\int_0^{\infty}\!\int_0^{\infty}\frac{\mathrm{e}^{-(m+n)s}}{\sqrt{1-\mathrm{e}^{-2s}}}\frac{\mathrm{e}^{-nt}}{\sqrt{1-\mathrm{e}^{-2t}}}\mathbb{E}\bigg[(\Sigma^{-1/2}\mathbf{Z})_{i_k}(\Sigma^{-1/2}\mathbf{Z}')_{i_l}\nonumber\\
\label{int23}&\quad\times\frac{\partial^{m+n-2}(h\circ g)}{\prod_{\stackrel{1\leq j\leq m+n-2}{j\not=k,l}}\partial w_{i_j}}(\mathbf{z}_{s,t,\mathbf{w}}^{\Sigma^{1/2}\mathbf{Z},\Sigma^{1/2}\mathbf{Z}'})\bigg]\,\mathrm{d}t\,\mathrm{d}s.
\end{align}
The desired bounds now follow from (\ref{int13}) and (\ref{int23}) and Lemma \ref{bell lem}.  
\end{proof}

\begin{proposition}Suppose $\Sigma$ is non-negative definite and that $h\in C_b^{m+n}(\mathbb{R})$ and $g\in C_P^{m+n}(\mathbb{R}^d)$, where $m,n\geq 1$ and $P\in\mathcal{F}$.  Then, for all $\mathbf{w}\in\mathbb{R}^d$,
\begin{align}\label{psip1}\bigg|\frac{\partial^n\psi_m(\mathbf{w})}{\prod_{j=1}^n\partial w_{i_j}}\bigg|&\leq\frac{h_{m+n}}{n(m+n)}\Big[\alpha+\beta_3\big(2\mathbb{E}P_1(\Sigma^{1/2}\mathbf{Z})+P_1(\mathbf{w})\big)+\gamma_3\big(\mathbb{E}P_2(\delta_3\Sigma^{1/2}\mathbf{Z})\big)^2P_2(\delta_3\mathbf{w})\Big].
\end{align}
Suppose now that $\Sigma$ is positive definite and that $h\in C_b^{m+n-2}(\mathbb{R})$ and $g\in C_P^{m+n-2}(\mathbb{R}^d)$, where $m,n\geq 1$ with $m+n\geq 3$ and $P\in\mathcal{F}_*$.  Then, for all $\mathbf{w}\in\mathbb{R}^d$,
\begin{align}
\bigg|\frac{\partial^{n}\psi_{m}(\mathbf{w})}{\prod_{j=1}^{n}\partial w_{i_j}}\bigg|&\leq2h_{m+n-2}\min_{1\leq k,l\leq d}\Big[\alpha\mathbb{E}|(\Sigma^{-1/2}\mathbf{Z})_k|\mathbb{E}|(\Sigma^{-1/2}\mathbf{Z})_l|\nonumber\\
&\quad+\beta_3\big(\mathbb{E}|(\Sigma^{-1/2}\mathbf{Z})_kP_1(\Sigma^{1/2}\mathbf{Z})|\mathbb{E}|(\Sigma^{-1/2}\mathbf{Z})_l|+\mathbb{E}|(\Sigma^{-1/2}\mathbf{Z})_k|\nonumber\\
&\quad\times\mathbb{E}|(\Sigma^{-1/2}\mathbf{Z})_lP_1(\Sigma^{1/2}\mathbf{Z})|+\mathbb{E}|(\Sigma^{-1/2}\mathbf{Z})_k|\mathbb{E}|(\Sigma^{-1/2}\mathbf{Z})_l|P_1(\mathbf{w})\big)\nonumber\\
\label{psip2}&\quad+\gamma_3P_2(\delta_3\mathbf{w})\mathbb{E}|(\Sigma^{-1/2}\mathbf{Z})_kP_2(\delta_3\Sigma^{1/2}\mathbf{Z})\mathbb{E}|(\Sigma^{-1/2}\mathbf{Z})_lP_2(\delta_3\Sigma^{1/2}\mathbf{Z})|\Big].
\end{align}
\end{proposition}

\begin{proof}By using a similar argument to the one used to prove inequality (\ref{pinequality}) we obtain
\begin{align}\label{pin22}P(\mathbf{z}_{s,t,\mathbf{w}}^{\Sigma^{1/2}\mathbf{Z},\Sigma^{1/2}\mathbf{Z}'})
\leq \alpha+\beta_3\big(P_1(\mathbf{w})+P_1(\Sigma^{1/2}\mathbf{Z})+P_1(\Sigma^{1/2}\mathbf{Z}')\big)
+\gamma_3P_2(\delta_3\mathbf{w})P_2(\delta_3\Sigma^{1/2}\mathbf{Z})P_2(\delta_3\Sigma^{1/2}\mathbf{Z}').
\end{align}
We then proceed as we did in the proof of Proposition \ref{btlet} by substituting (\ref{pin22}) into the integral inequalities of Lemma \ref{lempppss} and then bounding the resulting integrals.  Here, in obtaining (\ref{psip2}) we used the inequality 
\begin{align*}\int_0^{\infty}\!\int_0^{\infty}\frac{\mathrm{e}^{-(m+n)s}}{\sqrt{1-\mathrm{e}^{-2s}}}\frac{\mathrm{e}^{-nt}}{\sqrt{1-\mathrm{e}^{-2t}}}\,\mathrm{d}s\,\mathrm{d}t&\leq \int_0^{\infty}\frac{\mathrm{e}^{-3s}}{\sqrt{1-\mathrm{e}^{-2s}}}\,\mathrm{d}s\int_0^{\infty}\frac{\mathrm{e}^{-t}}{\sqrt{1-\mathrm{e}^{-2t}}}\,\mathrm{d}t=\frac{\pi}{4}\cdot\frac{\pi}{2}<2,
\end{align*}
which holds since $n\geq 1$ and $m+n\geq 3$.
\end{proof}

Again, in the univariate case it is possible to obtain a bound for the partial derivatives of $\psi_m$ that involve fewer derivatives of $h$ and $g$.

\begin{proposition}Fix $d=1$ and let $\Sigma=1$.  Let $m\geq 2$ and suppose that $h\in C_b^{m-1}(\mathbb{R})$ and $g\in C_P^{m-1}(\mathbb{R}^d)$, where $P\in\mathcal{F}_*$.  Then, for all $w\in\mathbb{R}$,
\begin{align*}|\psi_{m}^{(3)}(w)|&\leq h_{m-1}\Big[2\alpha(3+w^2)+P_1(w)+P_2(w) +\beta_2\big((1+2|w|)P_1(w)+\mathbb{E}P_1(Z)+2|w|\mathbb{E}|ZP_1(Z)|\big) \\
&\quad+2\beta_3(1+w^2)\big(P_1(w)+2\mathbb{E}|ZP_1(Z)|\big) +\gamma_2\big(P_2(\delta_2w)\mathbb{E}P_2(\delta_2Z)+2|wP_2(\delta_2w)|\mathbb{E}|ZP_2(\delta_2Z)|\big) \\
&\quad+2\gamma_3(1+w^2)P_2(\delta_3w)[\mathbb{E}|ZP_2(\delta_3Z)|]^2\Big].
\end{align*}
\end{proposition}

\begin{proof}The solution $\psi_m$ satisfies the Stein equation $\psi_m''(w)-w\psi_m'(w)=f^{(m)}(w)$, and therefore
\begin{align*}|\psi_m^{(3)}(w)|&=|f^{(m+1)}(w)-w\psi_m''(w)-\psi_m'(w)| =|f^{(m+1)}(w)-wf^{(m)}(w)+w^2\psi'(w)-\psi'(w)| \\
&\leq|f^{(m+1)}(w)|+|wf^{(m)}(w)|+(1+w^2)|\psi_{m}'(w)|.
\end{align*}
Bounding the final three terms using (\ref{f2soln}), (\ref{wfirst}) and (\ref{psip2}), and simplifying the resulting bound by using that $\mathbb{E}|Z|<1$ completes the proof.
\end{proof}

\subsection{Bounds for polynomial and exponential $P$}\label{sec2point3}
In Section \ref{sec2point2}, we gave bounds for the derivatives of $f$ and $\psi_m$ in terms of a dominating function $P$ from a general class of functions $\mathcal{F}$ or $\mathcal{F}_*$.  As was noted in Examples \ref{ex1} and \ref{ex2}, the functions $P(\mathbf{w})=A+B\sum_{i=1}^d|w_i|^{r_i}$ and $P(\mathbf{w})=A\exp(a\sum_{i=1}^d|w_i|^b)$ are contained in these classes.  Therefore we can obtain bounds for the derivatives of $f$ and $\psi_m$ for the case that the derivatives of $g$ have polynomial or exponential growth as special cases of the bounds of Section \ref{sec2point2}.  The bounds for the case of polynomial $P$ will be used in the proofs of Theorems \ref{winfirst1}--\ref{multievenguni}.  The bounds for the case of exponential $P$ will not be further used in this paper, but may prove useful in other applications; for a further discussion see Remark \ref{rem3.6}.

\begin{corollary}\label{cor28}Let $P(\mathbf{w})=A+B\sum_{i=1}^d|w_i|^{r_i}$, where $r_i\geq 0$, $i=1,\ldots,d$.  Suppose $\Sigma$ is non-negative definite and $h\in C_b^n(\mathbb{R})$ and $g\in C_P^n(\mathbb{R}^d)$ for $n\geq 1$.  Let $Z_i=(\Sigma^{1/2}\mathbf{Z})_i\sim N(0,\sigma_{ii}^2)$.  Then, for all $\mathbf{w}\in\mathbb{R}^d$,
\begin{align*}\bigg|\frac{\partial^nf(\mathbf{w})}{\prod_{j=1}^n\partial w_{i_j}}\bigg|&\leq\frac{h_n}{n}\bigg[A+B\sum_{i=1}^d2^{r_i}\big(|w_i|^{r_i}+\mathbb{E}|Z_i|^{r_i}\big)\bigg].
\end{align*}
Suppose now that $\Sigma$ is positive definite and $h\in C_b^{n-1}(\mathbb{R})$ and $g\in C_P^{n-1}(\mathbb{R}^d)$ for $n\geq 2$.  Then, for all $\mathbf{w}\in\mathbb{R}^d$,
\begin{align*}
\bigg|\frac{\partial^{n}f(\mathbf{w})}{\prod_{j=1}^{n}\partial w_{i_j}}\bigg|&\leq h_{n-1}\min_{1\leq l\leq d}
\bigg[A\mathbb{E}|(\Sigma^{-1/2}\mathbf{Z})_l|+B\sum_{i=1}^d2^{r_i}\big(|w_i|^{r_i}\mathbb{E}|(\Sigma^{-1/2}\mathbf{Z})_l|+\mathbb{E}|(\Sigma^{-1/2}\mathbf{Z})_lZ_i^{r_i}|\big)\bigg].
\end{align*}
Suppose now that $\Sigma=I_d$.  Then we have the simplified bound
\begin{align*}\bigg|\frac{\partial^{n}f(\mathbf{w})}{\prod_{j=1}^{n}\partial w_{i_j}}\bigg|\leq h_{n-1}
\bigg[A+B\sum_{i=1}^d2^{r_i}\big(|w_i|^{r_i}+\mathbb{E}|Z|^{r_i+1}\big)\bigg].
\end{align*}
Consider now the case $d=1$ with $\Sigma=1$.  Suppose $h\in C_b^{n-2}(\mathbb{R})$ and $g\in C_P^{n-2}(\mathbb{R})$, where $n\geq 3$ and $P(w)=A+B|w|^r$, $r\geq0$.  Then, for all $w\in\mathbb{R}$,
\begin{align*}|f^{(n)}(w)|&\leq h_{n-2}\Big[3A+2^rB\big(|w|^{r+1}+2|w|^r+|w|\mathbb{E}|Z|^{r+1}+\mathbb{E}|Z|^r\big))\Big].
\end{align*}
\end{corollary}

\begin{proof}The bounds follow from applying inequalities (\ref{btlet1}), (\ref{btlet2}) and (\ref{f2soln}) with $P(\mathbf{w})=A+B\sum_{i=1}^d|w_i|^{r_i}$.  From Example \ref{ex1}, we have $\alpha=A$, $\beta_k=k^{r_*}$, where $r_*=\max_{1\leq i\leq d}r_i$, $\gamma_k=0$ and $P_1(\mathbf{w})=B\sum_{i=1}^d|w_i|^{r_i}$.  Although, by examining the derivations of inequalities (\ref{btlet1}), (\ref{btlet2}) and (\ref{f2soln}), we see that we can slightly improve on these bounds by using the inequality $P_1(\mathbf{w}_1+\cdots+\mathbf{w}_k)\leq \sum_{j=1}^kk^{r_i}P_1(\mathbf{w}_j)$, instead of inequality (\ref{ex1eqn}).   Finally, we simplify the final two bounds by using that $\mathbb{E}|Z|<1$. 
\end{proof}

\begin{corollary}\label{corpp}Let $P(\mathbf{w})=A+B\sum_{i=1}^d|w_i|^{r_i}$, where $r_i\geq 0$, $i=1,\ldots,d$.  Suppose $\Sigma$ is non-negative definite and $h\in C_b^{m+n}(\mathbb{R})$ and $g\in C_P^{m+n}(\mathbb{R}^d)$ for $m,n\geq 1$.  Then, for all $\mathbf{w}\in\mathbb{R}^d$,
\begin{align*}\bigg|\frac{\partial^n\psi_m(\mathbf{w})}{\prod_{j=1}^n\partial w_{i_j}}\bigg|&\leq\frac{h_{m+n}}{n(m+n)}\bigg[A+B\sum_{i=1}^d3^{r_i}\big(|w_i|^{r_i}+2\mathbb{E}|Z_i|^{r_i}\big)\bigg].
\end{align*}
Suppose now that $\Sigma$ is positive definite and $h\in C_b^{m+n-2}(\mathbb{R})$ and $g\in C_P^{m+n-2}(\mathbb{R}^d)$ for $m,n\geq 1$ and $m+n\geq 3$.  Then, for all $\mathbf{w}\in\mathbb{R}^d$,
\begin{align*}
\bigg|\frac{\partial^{n}\psi_{m}(\mathbf{w})}{\prod_{j=1}^{n}\partial w_{i_j}}\bigg|&\leq h_{m+n-2}\min_{1\leq k,l\leq d}\bigg[A\mathbb{E}|(\Sigma^{-1/2}\mathbf{Z})_k|\mathbb{E}|(\Sigma^{-1/2}\mathbf{Z})_l|\\
&\quad+B\sum_{i=1}^d3^{r_i}\big(|w_i|^{r_i}\mathbb{E}|(\Sigma^{-1/2}\mathbf{Z})_l|\mathbb{E}|(\Sigma^{-1/2}\mathbf{Z})_k|+2\mathbb{E}|(\Sigma^{-1/2}\mathbf{Z})_k|\mathbb{E}|(\Sigma^{-1/2}\mathbf{Z})_lZ_i^{r_i}|\big)\bigg].
\end{align*}
Suppose now that $\Sigma=I_d$.  Then we have the simplified bound
\begin{equation}\label{sigmaiii} \bigg|\frac{\partial^{n}\psi_{m}(\mathbf{w})}{\prod_{j=1}^{n}\partial w_{i_j}}\bigg|\leq h_{m+n-2}\bigg[A+B\sum_{i=1}^d3^{r_i}\big(|w_i|^{r_i}+2\mathbb{E}|Z|^{r_i+1}\big)\bigg].
\end{equation}
Consider now the case $d=1$ with $\Sigma=1$.  Suppose $h\in C_b^{m-1}(\mathbb{R})$ and $g\in C_P^{m-1}(\mathbb{R})$, where $m\geq 2$ and $P(w)=A+B|w|^r$, $r\geq0$.  Then, for all $w\in\mathbb{R}$,
\begin{align*}|\psi_m^{(3)}(w)|\leq h_{m-1}\Big[A(6+w^2) +2\cdot3^rB\big(|w|^{r+2}+2|w|^{r+1}+2|w|^r+2\mathbb{E}|Z|^{r+1}(1+|w|+w^2)+\mathbb{E}|Z|^r\big)\Big].
\end{align*}
\end{corollary} 

The proof of Corollary \ref{corpp} is analogous to that of Corollary \ref{cor28} and is omitted.  Similarly, one can obtain bounds for the case that the dominating function $P$ grows exponentially.  

\begin{corollary}Let $P(\mathbf{w})=A\exp(a\sum_{i=1}^d|w_i|^b)$, where $a\geq0$ and $b_i\geq0$, $i=1,\ldots,d$.  For each of the below inequalities, $a$ and the $b_i$ must be such that the expectation in the upper bound exists.  (A simple sufficient condition for this to be the case is that $a\geq0$ and $\max_{1\leq i\leq d}b_i<2$). Suppose $\Sigma$ is non-negative definite and $h\in C_b^n(\mathbb{R})$ and $g\in C_P^n(\mathbb{R}^d)$ for $n\geq 1$.  
Then, for all $\mathbf{w}\in\mathbb{R}^d$,
\begin{align*}\bigg|\frac{\partial^nf(\mathbf{w})}{\prod_{j=1}^n\partial w_{i_j}}\bigg|&\leq\frac{Ah_n}{n}\exp\bigg(a\sum_{i=1}^dc_{2,b_i}|w_i|^{b_i}\bigg)\mathbb{E}\exp\bigg(a\sum_{i=1}^dc_{2,b_i}|(\Sigma^{-1/2}\mathbf{Z})_i|^{b_i}\bigg).
\end{align*}
Suppose now that $\Sigma$ is positive definite and $h\in C_b^{n-1}(\mathbb{R})$ and $g\in C_P^{n-1}(\mathbb{R}^d)$ for $n\geq 2$.  Then, for all $\mathbf{w}\in\mathbb{R}^d$,
\begin{align*}
\bigg|\frac{\partial^{n}f(\mathbf{w})}{\prod_{j=1}^{n}\partial w_{i_j}}\bigg|&\leq Ah_{n-1}\exp\bigg(a\sum_{i=1}^dc_{2,b_i}|w_i|^{b_i}\bigg)\min_{1\leq l\leq d}\mathbb{E}\bigg|(\Sigma^{-1/2}\mathbf{Z})_i\exp\bigg(a\sum_{i=1}^dc_{2,b_i}|(\Sigma^{-1/2}\mathbf{Z})_i|^{b_i}\bigg)\bigg|.
\end{align*}
Consider now the case $d=1$ with $\Sigma=1$.  Suppose $h\in C_b^{n-2}(\mathbb{R})$ and $g\in C_P^{n-2}(\mathbb{R})$, where $n\geq 3$ and $P(w)=A\exp(a|w|^b)$.  Then, for all $w\in\mathbb{R}$,
\begin{align*}|f^{(n)}(w)|&\leq Ah_{n-2}\exp(ac_{2,b}|w|^b)\Big[1+\mathbb{E}\exp(ac_{2,b}|Z|^b)+|w|\mathbb{E}|Z\exp(ac_{2,b}|Z|^b)|\Big].
\end{align*}
\end{corollary}

\begin{corollary}Let $P(\mathbf{w})=A\exp(a\sum_{i=1}^d|w_i|^{b_i})$, where $a\geq0$ and $b_i\geq0$, $i=1,\ldots,d$.  Suppose $\Sigma$ is non-negative definite and $h\in C_b^{m+n}(\mathbb{R})$ and $g\in C_P^{m+n}(\mathbb{R}^d)$ for $m,n\geq 1$.  Then, for all $\mathbf{w}\in\mathbb{R}^d$,
\begin{align*}\bigg|\frac{\partial^n\psi_m(\mathbf{w})}{\prod_{j=1}^n\partial w_{i_j}}\bigg|&\leq\frac{Ah_{m+n}}{n(m+n)}\exp\bigg(\!a\!\sum_{i=1}^d\!c_{3,b_i}|w_i|^{b_i}\!\bigg)\bigg\{\mathbb{E}\bigg(\!a\!\sum_{i=1}^d\!c_{3,b_i}|(\Sigma^{-1/2}\mathbf{Z})_i|^{b_i}\!\bigg)\bigg\}^2.
\end{align*}
Suppose now that $\Sigma$ is positive definite and $h\in C_b^{m+n-2}(\mathbb{R})$ and $g\in C_P^{m+n-2}(\mathbb{R}^d)$ for $m,n\geq 1$ and $m+n\geq 3$.  Then, for all $\mathbf{w}\in\mathbb{R}^d$,
\begin{align*}\bigg|\frac{\partial^{n}\psi_{m}(\mathbf{w})}{\prod_{j=1}^{n}\partial w_{i_j}}\bigg|&\leq Ah_{m+n-2}\exp\bigg(a\sum_{i=1}^dc_{3,b_i}|w_i|^{b_i}\bigg)\min_{1\leq l\leq d}\bigg\{\mathbb{E}\bigg|(\Sigma^{-1/2}\mathbf{Z})_l\exp\bigg(a\sum_{i=1}^dc_{3,b_i}|(\Sigma^{-1/2}\mathbf{Z})_i|^{b_i}\bigg)\bigg|\bigg\}^2.
\end{align*}
Consider now the case $d=1$ with $\Sigma=1$. Suppose $h\in C_b^{m-1}(\mathbb{R})$ and $g\in C_P^{m-1}(\mathbb{R})$, where $m\geq 2$ and $P(w)=A\exp(a|w|^b)$.  Then, for all $w\in\mathbb{R}$,
\begin{align*}|\psi_m^{(3)}(w)|&\leq Ah_{m-1}\exp(ac_{3,b}|w|^b)\Big[1+\mathbb{E}\exp(ac_{2,b}|Z|^b)\\
&\quad+2|w|\mathbb{E}|Z\exp(ac_{2,b}|Z|^b)|+2(1+w^2)\big\{\mathbb{E}|Z\exp(ac_{3,b}|Z|^b)|\big\}^2\Big].
\end{align*}
\end{corollary} 

\section{Bounds for the distributional distance between $g(\mathbf{W})$ and $g(\mathbf{Z})$}\label{sec3rd}

With the bounds for the derivatives of the solution of the multivariate normal Stein equation with test function $h(g(\cdot))$ stated in Section \ref{sec2nd}, we are in a position to obtain bounds for the distributional distance between $g(\mathbf{W})$ and its limiting distribution $g(\Sigma^{1/2}\mathbf{Z})$.  Such bounds can be achieved by bounding the expectation $\mathbb{E}[\nabla^T\Sigma\nabla f(\mathbf{W})-\mathbf{W}^T\nabla f(\mathbf{W})]$ by using various coupling techniques that have been developed for multivariate normal approximation (see \cite{goldstein 2, goldstein1, reinert 1, meckes}), where the coupling is chosen based on the dependence structure of $\mathbf{W}$.

For the rest of this paper, we shall consider the case that $\mathbf{W}=(W_1,\ldots,W_d)$, where $W_j=\frac{1}{\sqrt{n_j}}\sum_{i=1}^{n_j}X_{ij}$ and the $X_{ij}$ are mutually independent (as a result, in this section, we shall mostly be taking $\Sigma=I_d)$.  From here on, $\mathbf{W}$ will denote such a random vector.  The restriction to this class of statistics allows for a detailed investigation of convergence rates, and we would expect that the factors effecting convergence rates here (matching moments, whether $g$ is an even function, and the differentiability and growth rate of $g$) to also to apply in more general settings.  

\subsection{Preliminary lemmas}

We begin by obtaining bounds for the distributional distance between $g(\mathbf{W})$ and $g(\mathbf{Z})$ in terms of the derivatives of the solution of the $\mathrm{MVN}(\mathbf{0},\Sigma)$ Stein equation with test function $h(g(\cdot))$.  We give two bounds: one for general $g$ and another for when $g$ is an even function.  In Section \ref{sec3point2}, we apply these bounds and those of Section \ref{sec2point3} to bound the distance for the case that the derivatives of $g$ have polynomial growth.  

Unless otherwise stated, in this section, $f$ will denote the solution (\ref{mvnsolnh}).  We shall also let $C^k(\mathbb{R}^d)$ denote the case of real-valued functions defined on $\mathbb{R}^d$ whose partial derivatives of order $k$ all exist.  We define the random vector $\mathbf{X}_{ij}$ to be such that it has $X_{ij}$ as its $j$-th entry and the other $d-1$ entries are given by zero.  For all $1\leq i\leq n$ and $1\leq j\leq d$, we define $\mathbf{W}^{(i,j)}=\mathbf{W}-\frac{1}{\sqrt{n_j}}\mathbf{X}_{ij}$, so that $\mathbf{W}^{(i,j)}$ is independent of $\mathbf{X}_{ij}$.    

\begin{lemma}\label{noteveng}Let $X_{1,1},\ldots,X_{n_1,1},\ldots,X_{1,d},\ldots,X_{n_d,d}$ be independent random variables with $\mathbb{E}X_{ij}^k=\mathbb{E}Z^k$ for all $1\leq i\leq n_j$, $1\leq j\leq d$ and all positive integers $k\leq p$.  Let $\Sigma=I_d$ and suppose $h$ and $g$ are such that $f\in C_b^{p+1}(\mathbb{R}^d)$.  Then, if the expectations on the right-hand side of (\ref{springz}) exist, 
\begin{align} |\mathbb{E}h(g(\mathbf{W}))-\mathbb{E}h(g(\mathbf{Z}))|&\leq\sum_{j=1}^d\sum_{i=1}^{n_j}\frac{1}{(p-1)!n_j^{(p+1)/2}}\bigg\{\sup_{\theta}\mathbb{E}\bigg|X_{ij}^{p-1}\frac{\partial^{p+1}f}{\partial w_j^{p+1}}(\mathbf{W}_\theta^{(i,j)})\bigg|\nonumber\\
\label{springz}&\quad+\frac{1}{p}\sup_{\theta}\mathbb{E}\bigg|X_{ij}^{p+1}\frac{\partial^{p+1}f}{\partial w_j^{p+1}}(\mathbf{W}_\theta^{(i,j)})\bigg|\bigg\},
\end{align}
where $\mathbf{W}_\theta^{(i,j)}=\mathbf{W}^{(i,j)}+\frac{\theta}{\sqrt{n_j}}\mathbf{X}_{ij}$ for some $\theta\in (0,1)$. 
\end{lemma}

\begin{proof}We aim to bound $\mathbb{E}h(g(\mathbf{W}))-\mathbb{E}h(g(\mathbf{Z}))$, and do so by bounding the quantity $\sum_{j=1}^d\mathbb{E}\Big[\frac{\partial^2f}{\partial w_j^2}(\mathbf{W})-W_j\frac{\partial f}{\partial w_j}(\mathbf{W})\Big]$.    Taylor expanding $\frac{\partial^2f}{\partial w_j^2}(\mathbf{W})$ and $\frac{\partial f}{\partial w_j}(\mathbf{W})$ about $\mathbf{W}^{(i,j)}$ gives
\begin{align}\sum_{j=1}^d\mathbb{E}\bigg[\frac{\partial^2f}{\partial w_j^2}(\mathbf{W})-W_j\frac{\partial f}{\partial w_j}(\mathbf{W})\bigg]
&=\sum_{j=1}^d\sum_{i=1}^{n_j}\frac{1}{n_j}\mathbb{E}\frac{\partial^2f}{\partial w_j^2}(\mathbf{W})-\sum_{j=1}^d\sum_{i=1}^{n_j}\frac{1}{\sqrt{n_j}}\mathbb{E}X_{ij}\frac{\partial f}{\partial w_j}(\mathbf{W}) \nonumber\\
&=\sum_{j=1}^d\sum_{i=1}^{n_j}\sum_{k=0}^{p-2}\frac{1}{k!n_j^{k/2+1}}\mathbb{E}X_{ij}^k\frac{\partial^{k+2} f}{\partial w_j^{k+2}}(\mathbf{W}^{(i,j)})\nonumber \\
\label{rhsf}&\quad-\sum_{j=0}^d\sum_{i=1}^{n_j}\sum_{k=0}^{p-1}\frac{1}{k!n_j^{k/2+1/2}}\mathbb{E}X_{ij}^{k+1}\frac{\partial^{k+1} f}{\partial w_j^{k+1}}(\mathbf{W}^{(i,j)})+R_1+R_2,
\end{align}
where
\begin{eqnarray*}|R_{1}|&\leq&\sum_{j=1}^d\sum_{i=1}^{n_j}\frac{1}{(p-1)!n_j^{(p+1)/2}}\sup_{\theta}\mathbb{E}\bigg|X_{ij}^{p-1}\frac{\partial^{p+1} f}{\partial w_j^{p+1}}(\mathbf{W}_\theta^{(i,j)})\bigg|, \\
 |R_{2}|&\leq&\sum_{j=1}^d\sum_{i=1}^{n_j}\frac{1}{p!n_j^{(p+1)/2}}\sup_{\theta}\mathbb{E}\bigg|X_{ij}^{p+1}\frac{\partial^{p+1} f}{\partial w_j^{p+1}}(\mathbf{W}_\theta^{(i,j)})\bigg|.
\end{eqnarray*}
Using independence and that the $X_{ij}$ have mean zero and collecting terms, we can write the right-hand side of (\ref{rhsf}) as
\begin{align*}\sum_{j=1}^d\sum_{i=1}^{n_j}\sum_{k=1}^{p-1}\frac{1}{k!n_j^{k/2+1/2}}[k\mathbb{E}X_{ij}^{k-1}-\mathbb{E}X_{ij}^{k+1}]\mathbb{E}\frac{\partial^{k+1} f}{\partial w_j^{k+1}}(\mathbf{W}^{(i,j)})+R_1+R_2.
\end{align*}
Now, by the matching moments assumption, $k\mathbb{E}X_{ij}^{k-1}-\mathbb{E}X_{ij}^{k+1}=k\mathbb{E}Z^{k-1}-\mathbb{E}Z^{k+1}$ for all $1\leq k\leq p-1$.  But the moments of the standard normal distribution satisfy $k\mathbb{E}Z^{k-1}-\mathbb{E}Z^{k+1}=0$ for all $k>0$.  Thus, $|\mathbb{E}h(g(\mathbf{W}))-\mathbb{E}h(g(\mathbf{Z}))|\leq |R_1|+|R_2|$, and the proof is complete.
\end{proof}

\begin{remark}In the statement of Lemma \ref{noteveng}, we did not give precise conditions on $h$ and $g$ such that $f\in C_b^{p+1}(\mathbb{R}^d)$, nor restrictions on the $X_{ij}$ such that the expectations on the right-hand side of (\ref{springz}) exist.  In applying, Lemma \ref{noteveng} in practice (see Section \ref{sec3point2}), one would need to check that $h$, $g$ and the $X_{ij}$ are such that these conditions are met.  These comments apply equally to Lemmas \ref{evennormal}--\ref{evenglemmabound}. 
\end{remark}

We now turn our attention to the case that $g$ is an even function.  The following key lemma enables us to obtain faster convergence rates in this case.

\begin{lemma} \label{evennormal}Let $\Sigma$ be non-negative definite.   Suppose that $g:\mathbb{R}^d\rightarrow\mathbb{R}$ is an even function ($g(\mathbf{w})=g(-\mathbf{w})$ for all $\mathbf{w}\in\mathbb{R}^d$).  Then, the solution (\ref{mvnsolnh}), denoted by $f$, is an even function.  Moreover, for odd $k\geq 1$, provided that $\frac{\partial^kf(\mathbf{w})}{\prod_{j=1}^k\partial w_{i_j}}$ exists, 
\begin{equation}\label{evenexpectf}\mathbb{E}\bigg[\frac{\partial^kf(\Sigma^{1/2}\mathbf{Z})}{\prod_{j=1}^k\partial w_{i_j}}\bigg]=0,
\end{equation}
if the expectation in (\ref{evenexpectf}) is well-defined.
\end{lemma}

\begin{proof}As $\Sigma^{1/2}\mathbf{Z}\stackrel{\mathcal{D}}{=}-\Sigma^{1/2}\mathbf{Z}$ and $g$ is an even function, we have, for any $\mathbf{w}\in\mathbb{R}^d$,
\begin{align*}f(-\mathbf{w})&=-\int_0^{\infty}\mathbb{E}\big[h(g(-\mathrm{e}^{-s}\mathbf{w}+\sqrt{1-\mathrm{e}^{-2s}}\Sigma^{1/2}\mathbf{Z}))-\mathbb{E}h(g(\Sigma^{1/2}\mathbf{Z}))\big]\,\mathrm{d}s \\
&=-\int_0^{\infty}\big[h(g(-\mathrm{e}^{-s}\mathbf{w}-\sqrt{1-\mathrm{e}^{-2s}}\Sigma^{1/2}\mathbf{Z}))-\mathbb{E}h(g(\Sigma^{1/2}\mathbf{Z}))\big]\,\mathrm{d}s \\
&=-\int_0^{\infty}\big[h(g(\mathrm{e}^{-s}\mathbf{w}+\sqrt{1-\mathrm{e}^{-2s}}\Sigma^{1/2}\mathbf{Z}))-\mathbb{E}h(g(\Sigma^{1/2}\mathbf{Z}))\big]\,\mathrm{d}s =f(\mathbf{w}),
\end{align*}
and therefore the solution (\ref{mvnsoln}) is an even function.

Since $f$ is an even function, the partial derivatives of odd order are odd functions, provided they exist.  Therefore, since $\Sigma^{1/2}\mathbf{Z}\stackrel{\mathcal{D}}{=}-\Sigma^{1/2}\mathbf{Z}$, it follows that (\ref{evenexpectf}) holds.
\end{proof} 

With the aid of Lemma \ref{evennormal}, we are able to obtain an analogue of Lemma \ref{noteveng} for the case that $g$ is an even function.  The symmetry of the function $g$ allows us to obtain faster convergence rates, for smooth test functions $h$.  The following partial differential equation shall appear in our proof:
\begin{equation} \label{234multinor} \sum_{k=1}^d\bigg(\frac{\partial^2 \psi_j}{\partial w_k^2}(\mathbf{w})-w_k\frac{\partial \psi_j}{\partial w_k}(\mathbf{w})\bigg)=\frac{\partial^{p+1} f}{\partial w_j^{p+1}}(\mathbf{w}).
\end{equation}
Bounds for the solution $\psi_j$ and its partial derivatives were given in Sections 2.2 and 2.3.

\begin{lemma}\label{evenglemmabound}Let $X_{1,1},\ldots,X_{n_1,1},\ldots,X_{1,d},\ldots,X_{n_d,d}$ be independent random variables with $\mathbb{E}X_{ij}^k=\mathbb{E}Z^k$ for all $1\leq i\leq n_j$, $1\leq j\leq d$ and all positive integers $k\leq p$.  Suppose $g:\mathbb{R}^d\rightarrow\mathbb{R}$ is an even function.  Suppose further that the solution (\ref{mvnsolnh}), denoted by $f$, belongs to the class $C^{p+2}(\mathbb{R}^d)$ and that the solution $\psi_j$ to (\ref{234multinor}) is in the class $C^3(\mathbb{R}^d)$.  Then, if the expectations on the right-hand side of (\ref{dig hole}) exist, 
\begin{align}&|\mathbb{E}h(g(\mathbf{W}))-\mathbb{E}h(g(\mathbf{Z}))| \leq\sum_{j=1}^d\sum_{i=1}^{n_j}\frac{1}{p!n_j^{p/2+1}}\bigg\{\sup_{\theta}\mathbb{E}\bigg|X_{ij}^{p}\frac{\partial^{p+2}f}{\partial w_j^{p+2}}(\mathbf{W}_\theta^{(i,j)})\bigg|\nonumber\\
&\quad+\!\frac{1}{p+1}\!\sup_{\theta}\mathbb{E}\bigg|X_{ij}^{p+2}\frac{\partial^{p+2}f}{\partial w_j^{p+2}}(\mathbf{W}_\theta^{(i,j)})\bigg| \!+\!|\mathbb{E}X_{ij}^{p+1}|\sup_{\theta}\mathbb{E}\bigg|X_{ij}\frac{\partial^{p+2}f}{\partial w_j^{p+2}}(\mathbf{W}_\theta^{(i,j)})\bigg|\bigg\} \nonumber\\
\label{dig hole}& \quad+\sum_{j=1}^d\sum_{i=1}^{n_j}\frac{|\mathbb{E}X_{ij}^{p+1}|}{p!n_j^{(p+1)/2}}\sum_{k=1}^d\sum_{l=1}^{n_k}\frac{1}{n_k^{3/2}}\bigg\{\sup_{\theta}\mathbb{E}\bigg|X_{lk}\frac{\partial^3\psi_j}{\partial w_k^3}(\mathbf{W}_{\theta}^{(l,k)})\bigg|+\frac{1}{2}\sup_{\theta}\mathbb{E}\bigg|X_{lk}^3\frac{\partial^3\psi_j}{\partial w_k^3}(\mathbf{W}_{\theta}^{(l,k)})\bigg|\bigg\},
\end{align}
where $\mathbf{W}_\theta^{(i,j)}$ is defined in Lemma \ref{noteveng}.
\end{lemma}

\begin{proof}By a similar argument to the one used in the proof of Lemma \ref{noteveng}, we have
\begin{align*}\sum_{j=1}^d\mathbb{E}\bigg[\frac{\partial^2f}{\partial w_j^2}(\mathbf{W})-W_j\frac{\partial f}{\partial w_j}(\mathbf{W})\bigg]=\sum_{j=1}^d\sum_{i=1}^{n_j}\frac{1}{p!n_j^{p/2+1/2}}\big[p\mathbb{E}X_{ij}^{p-1}-\mathbb{E}X_{ij}^{p+1}\big]\mathbb{E}\frac{\partial^{p+1} f}{\partial w_j^{p+1}}(\mathbf{W}^{(i,j)})+R_1+R_2,
\end{align*}
where
\begin{eqnarray*}|R_{1}|&\leq&\sum_{j=1}^d\sum_{i=1}^{n_j}\frac{1}{p!n_j^{p/2+1}}\sup_{\theta}\mathbb{E}\bigg|X_{ij}^{p}\frac{\partial^{p+2} f}{\partial w_j^{p+2}}(\mathbf{W}_\theta^{(i,j)})\bigg|, \\
 |R_{2}|&\leq&\sum_{j=1}^d\sum_{i=1}^{n_j}\frac{1}{(p+1)!n_j^{p/2+1}}\sup_{\theta}\mathbb{E}\bigg|X_{ij}^{p+2}\frac{\partial^{p+2} f}{\partial w_j^{p+2}}(\mathbf{W}_\theta^{(i,j)})\bigg|.
\end{eqnarray*}
By the matching moments assumption, $\mathbb{E}X_{ij}^{p-1}=\mathbb{E}Z^{p-1}=0$.  Using this fact and Taylor expanding $\frac{\partial^{p+1} f}{\partial w_j^{p+1}}(\mathbf{W}^{(i,j)})$ about $\mathbf{W}$ gives
\begin{align*}\sum_{j=1}^d\!\mathbb{E}\bigg[\frac{\partial^2f}{\partial w_j^2}(\mathbf{W})\!-\!W_j\frac{\partial f}{\partial w_j}(\mathbf{W})\bigg]
&=-\!\sum_{j=1}^d\sum_{i=1}^{n_j}\!\frac{1}{p!n_j^{(p+1)/2}}\mathbb{E}X_{ij}^{p+1}\mathbb{E}\frac{\partial^{p+1} f}{\partial w_j^{p+1}}(\mathbf{W}^{(i,j)})+R_1+R_2 \\
&= N+R_1+R_2+R_3,
\end{align*}
where
\begin{eqnarray*}N&=&-\sum_{j=1}^d\sum_{i=1}^{n_j}\frac{1}{p!n_j^{(p+1)/2}}\mathbb{E}X_{ij}^{p+1}\mathbb{E}\frac{\partial^{p+1} f}{\partial w_j^{p+1}}(\mathbf{W}),\\
|R_3|&\leq&\sum_{j=1}^d\sum_{i=1}^{n_j}\frac{1}{p!n_j^{p/2+1}}|\mathbb{E}X_{ij}^{p+1}|\sup_{\theta}\mathbb{E}\bigg|X_{ij}\frac{\partial^{p+2} f}{\partial w_j^{p+2}}(\mathbf{W}_\theta^{(i,j)})\bigg|.
\end{eqnarray*}
To achieve the desired $O(n_1^{-p/2}+\cdots+n_d^{-p/2})$ bound we need to show that $\mathbb{E}\frac{\partial^{p+1} f}{\partial w_j^{p+1}}(\mathbf{W})$ is of order $n_1^{-1/2}+\cdots+n_d^{-1/2}$, since in general $\mathbb{E}X_{ij}^{p+1}\not=0$.  We consider the $\mathrm{MVN}(\mathbf{0},I_d)$ Stein equation with test function $\frac{\partial^{p+1} f}{\partial w_j^{p+1}}$:
\begin{equation*}\sum_{k=1}^d\bigg(\frac{\partial^2 \psi_j}{\partial w_k^2}(\mathbf{w})-w_k\frac{\partial \psi_j}{\partial w_k}(\mathbf{w})\bigg)=\frac{\partial^{p+1} f}{\partial w_j^{p+1}}(\mathbf{w})-\mathbb{E}\bigg[\frac{\partial^{p+1} f}{\partial w_j^{p+1}}(\mathbf{Z})\bigg].
\end{equation*}
By Lemma \ref{evennormal}, we have that $\mathbb{E}\frac{\partial^{p+1} f}{\partial w_j^{p+1}}(\mathbf{Z})=0$, and therefore
\begin{equation}\label{mpat}\mathbb{E}\bigg[\frac{\partial^{p+1}f}{\partial w_j^{p+1}}(\mathbf{W})\bigg]=\sum_{k=1}^d\mathbb{E}\bigg[\frac{\partial^2 \psi_j}{\partial w_k^2}(\mathbf{W})-W_k\frac{\partial \psi_j}{\partial w_k}(\mathbf{W})\bigg].
\end{equation}
We can use Lemma \ref{noteveng} to bound the right-hand side of (\ref{mpat}), which allows us to bound $N$.  All terms have now been bounded to the desired order and the proof is complete. 
\end{proof}

\subsection{Approximation theorems for polynomial $P$}\label{sec3point2}

Lemmas \ref{noteveng} and \ref{evenglemmabound} allow one to bound the distributional distance between $g(\mathbf{W})$ and $g(\mathbf{Z})$ if bounds are available for the expectations on the right-hand side of (\ref{springz}) and (\ref{dig hole}), respectively.  In this section, we obtain such bounds for the case that the derivatives of $g$ have polynomial growth.  We do not give bounds for the case of $g$ with derivatives of exponential growth, but see Remark \ref{rem3.6} for a further discussion.  We begin by proving the following lemma.

\begin{lemma}\label{cbwbshc}Let $P(\mathbf{w})=A+B\sum_{i=1}^d|w_i|^{r_i}$, where $A$, $B$ and $r_1,\ldots,r_d$ are non-negative constants.  Suppose $\Sigma=I_d$, $\theta\in(0,1)$ and let $q\geq 0$.  Then
\begin{align}\mathbb{E}\bigg|X_{ij}^q\frac{\partial^pf}{\partial w_j^p}(\mathbf{W}_\theta^{(i,j)})\bigg| &\leq \frac{h_p}{p}\bigg[A\mathbb{E}|X_{ij}|^q+B\sum_{k=1}^d2^{r_k}\bigg(2^{r_k}\mathbb{E}|X_{ij}|^q\mathbb{E}|W_k|^{r_k} \nonumber \\
&\quad+\frac{2^{r_k}}{n_k^{r_k/2}}\mathbb{E}|X_{ij}^qX_{ik}^{r_k}|+\mathbb{E}|Z|^{r_k}\mathbb{E}|X_{ij}|^{q}\bigg)\bigg], \nonumber \\
\mathbb{E}\bigg|X_{ij}^q\frac{\partial^pf}{\partial w_j^p}(\mathbf{W}_\theta^{(i,j)})\bigg| &\leq h_{p-1}\bigg[A\mathbb{E}|X_{ij}|^q+B\sum_{k=1}^d2^{r_k}\bigg(2^{r_k}\mathbb{E}|X_{ij}|^q\mathbb{E}|W_k|^{r_k} \nonumber \\
\label{cbhsxx}&\quad+\frac{2^{r_k}}{n_k^{r_k/2}}\mathbb{E}|X_{ij}^qX_{ik}^{r_k}|+\mathbb{E}|Z|^{r_k+1}\mathbb{E}|X_{ij}|^{q}\bigg)\bigg], \\
\mathbb{E}\bigg|X_{ij}^q\frac{\partial^3\psi_m}{\partial w_j^3}(\mathbf{W}_\theta^{(i,j)})\bigg| &\leq h_{m+1}\bigg[A\mathbb{E}|X_{ij}|^q+B\sum_{k=1}^d3^{r_k}\bigg(2^{r_k}\mathbb{E}|X_{ij}|^q\mathbb{E}|W_k|^{r_k} \nonumber \\
&\quad+\frac{2^{r_k}}{n_k^{r_k/2}}\mathbb{E}|X_{ij}^qX_{ik}^{r_k}|+2\mathbb{E}|Z|^{r_k+1}\mathbb{E}|X_{ij}|^{q}\bigg)\bigg], \nonumber
\end{align}
where the inequalities are for $g$ in the classes $C_P^p(\mathbb{R}^d)$, $C_P^{p-1}(\mathbb{R}^d)$ and $C_P^{m+1}(\mathbb{R}^d)$, respectively.  Suppose now that $d=1$ and $\Sigma=1$.  Then
\begin{align*}\mathbb{E}|X_i^qf^{(p)}(W_{\theta}^{(i)})|&\leq h_{p-2}\bigg[3A\mathbb{E}|X_{i}|^q+2^rB\bigg(2^{r+1}\mathbb{E}|X_i|^q\big(\mathbb{E}|W|^{r+1}+\mathbb{E}|W|^{r}\big)\\
&\quad+4\mathbb{E}|Z|^{r+1}\mathbb{E}|X_i|^{q+1}+\frac{2^{r+2}}{n^{r/2}}\mathbb{E}|X_i|^{r+q+1}\bigg)\bigg], \\
\mathbb{E}|X_i^q\psi_m^{(3)}(W_{\theta}^{(i)})|&\leq h_{m-1}\bigg[10A\mathbb{E}|X_i|^q+3^{r+1}B\bigg(2^{r+1}\mathbb{E}|X_i|^q\big(2\mathbb{E}|W|^{r+2}\\
&\quad+\mathbb{E}|W|^{r}\big) +16\mathbb{E}|Z|^{r+1}\mathbb{E}|X_i|^{q+2}+\frac{2^{r+3}}{n^{r/2}}\mathbb{E}|X_i|^{r+q+2}\bigg)\bigg],
\end{align*}
where the inequalities are for $g$ in the classes $C_P^{p-2}(\mathbb{R})$ and $C_P^{m-1}(\mathbb{R})$, respectively.  Here, $W_\theta^{(i)}=W^{(i)}+\frac{\theta}{\sqrt{n}}X_{i}$, where $W^{(i)}=W-\frac{1}{\sqrt{n}}X_i$.
\end{lemma}

\begin{proof}Let us prove the first inequality.  From inequality (\ref{sigmaiii}) we have
\begin{align*}\mathbb{E}\bigg|X_{ij}^q\frac{\partial^nf}{\partial w_j^n}(\mathbf{W}_{\theta}^{(i,j)})\bigg| &\leq \frac{h_n}{n}\bigg[A\mathbb{E}|X_{ij}|^q+B\sum_{k=1}^d2^{r_k}\Big(\mathbb{E}|X_{ij}^q((\mathbf{W}_{\theta}^{(i,j)})_k)^{r_k}|+\mathbb{E}|Z|^{r_k}\mathbb{E}|X_{ij}|^{q}\Big)\bigg],
\end{align*}
where $(\mathbf{W}_{\theta}^{(i,j)})_k$ denotes the $k$-th component of $\mathbf{W}_{\theta}^{(i,j)}$.  Note that $(\mathbf{W}_{\theta}^{(i,j)})_j=W_j^{(i)}+\frac{\theta}{\sqrt{n_j}}X_{ij}$ and that $(\mathbf{W}_{\theta}^{(i,j)})_k=W_k=W_k^{(i)}+\frac{1}{\sqrt{n_k}}X_{ik}$ for $k\not=j$, where $W_j^{(i)}=W_j-\frac{1}{\sqrt{n_j}}X_{ij}$.  Now, let $\theta_j=\theta\in(0,1)$, and $\theta_k=1$ if $k\not=j$.  By using the crude inequality $|a+b|^s\leq 2^s(|a|^s+|b|^s)$, which holds for any $s\geq0$, and independence of $X_{ij}$ and $W_k^{(i)}$, we have, for all $k=1,\ldots,d$, 
\begin{align}\mathbb{E}|X_{ij}^q(\mathbf{W}_{\theta}^{(i,j)})_k)^{r_k}|&\leq 2^{r_k}\mathbb{E}\bigg|X_{ij}^q\bigg(|W_k^{(i)}|^{r_k}+\frac{\theta_k^{r_k}}{n_k^{r_k/2}}|X_{ik}|^{r_k}\bigg)\bigg| \nonumber\\
\label{foxsell}&\leq 2^{r_k}\bigg(\mathbb{E}|X_{ij}|^q\mathbb{E}|W_k^{(i)}|^{r_k}+\frac{1}{n_k^{r_k/2}}\mathbb{E}|X_{ij}^qX_{ik}^{r_k}|\bigg).
\end{align}
Using that $\mathbb{E}|W_k^{(i)}|^{r_k}\leq \mathbb{E}|W_k|^{r_k}$ leads to the desired inequality.  This can be seen by using Jensen's inequality: 
\begin{align*}\mathbb{E}|W_k|^{r_k}&=\mathbb{E}[\mathbb{E}[|W_k+n_k^{-1/2}X_{ik}|^{r_k} \: | \: W_k^{(i)}]] \geq\mathbb{E}|\mathbb{E}[W_k+n_k^{-1/2}X_{ik} \: | \: W_{k}^{(i)}]|^{r_k} 
=\mathbb{E}|W_k^{(i)}|^{r_k}.
\end{align*}

The proofs of the other inequalities are similar, with the only difference being that for the final two inequalities in which $d=1$ we have
\begin{equation*}\mathbb{E}|X_{i}^q(W_{\theta}^{(i)})^{r+l}|\leq2^{r+l-1}\bigg(\mathbb{E}|X_{i}|^q\mathbb{E}|W^{(i)}|^{r+l}+\frac{1}{n^{r/2}}\mathbb{E}|X_{i}|^{q+r_k+l}\bigg), 
\end{equation*}
for $l=1,2$, which is obtained via an analogous calculation to the one used to obtain (\ref{foxsell}), but here we used the inequality $|a+b|^s\leq 2^{s-1}(|a|^s+|b|^s)$, which holds for any $s\geq1$.
\end{proof}

By applying the inequalities of Lemma \ref{cbwbshc} to the bounds of Lemmas \ref{noteveng} and \ref{evenglemmabound}, we can obtain the following four theorems for the distributional distance between $g(\mathbf{W})$ and $g(\mathbf{Z})$ when the derivatives of $g$ have polynomial growth.  Theorem \ref{winfirst1} follows from using inequality (\ref{cbhsxx}) in the bound of Lemma \ref{noteveng}, and the other theorems are proved similarly.

\begin{theorem}\label{winfirst1} Let $P(\mathbf{w})=A+B\sum_{i=1}^d|w_i|^{r_i}$, where $A$, $B$ and $r_1,\ldots,r_d$ are non-negative constants. Suppose $g\in C_P^p(\mathbb{R}^d)$. Let $X_{1,1},\ldots,X_{n,1},\ldots,X_{1,d},\ldots,X_{n,d}$ be independent random variables with $\mathbb{E}X_{ij}^k=\mathbb{E}Z^k$ for all $1\leq i\leq n_j$, $1\leq j\leq d$ and all positive integers $k\leq p$.
Suppose also that $\mathbb{E}|X_{ij}|^{r_l+p+1}<\infty$ for all $i$, $j$ and $l$.  Then, for $h\in C_b^p(\mathbb{R})$, 
\begin{align*} |\mathbb{E}h(g(\mathbf{W}))-\mathbb{E}h(g(\mathbf{Z}))|
&\leq\frac{p+1}{p!}h_p\sum_{j=1}^d\sum_{i=1}^{n_j}\frac{1}{n_j^{(p+1)/2}}\bigg[A\mathbb{E}|X_{ij}|^{p+1}+B\sum_{k=1}^d2^{r_k}\bigg(2^{r_k}\mathbb{E}|X_{ij}|^{p+1}\mathbb{E}|W_k|^{r_k} \\
&\quad+\frac{2^{r_k}}{n_k^{r_k/2}}\mathbb{E}|X_{ij}^{p+1}X_{ik}^{r_k}|+\mathbb{E}|Z|^{r_k+1}\mathbb{E}|X_{ij}|^{p+1}\bigg)\bigg].
\end{align*}
\end{theorem}

\begin{theorem}\label{winfirst2} Let $P(w)=A+B|w|^r$, where $A$, $B$ and $r$ are non-negative constants. Suppose $g\in C_P^{p-1}(\mathbb{R})$. Let $X_{1},\ldots,X_{n}$ be independent random variables with $\mathbb{E}X_{i}^k=\mathbb{E}Z^k$ for all $1\leq i\leq n$ and all positive integers $k\leq p$.  Suppose also that $\mathbb{E}|X_{i}|^{r+p+2}<\infty$ for all $1\leq i\leq n$.  Then, for $h\in C_b^{p-1}(\mathbb{R})$, 
\begin{align*}|\mathbb{E}h(g(W))-\mathbb{E}h(g(Z))|
&\leq\frac{p+1}{p!n^{(p+1)/2}}h_{p-1}\!\sum_{i=1}^{n}\bigg[3A\mathbb{E}|X_{i}|^{p+1}+2^rB\bigg(2^{r+1}\mathbb{E}|X_i|^{p+1}\big(\mathbb{E}|W|^{r+1}\!+\!\mathbb{E}|W|^{r}\big)\\
&\quad+4\mathbb{E}|Z|^{r+1}\mathbb{E}|X_i|^{p+2}+\frac{2^{r+2}}{n^{r/2}}\mathbb{E}|X_i|^{r+p+2}\bigg)\bigg].
\end{align*}
\end{theorem}

\begin{theorem}\label{multieveng} Let $P(\mathbf{w})=A+B\sum_{i=1}^d|w_i|^{r_i}$, where $A$, $B$ and $r_1,\ldots,r_d$ are non-negative constants. Suppose $g\in C_P^{p+2}(\mathbb{R}^d)$ is an even function. Let $X_{1,1},\ldots,X_{n,1},\ldots,X_{1,d},\ldots,X_{n,d}$ be independent random variables with $\mathbb{E}X_{ij}^k=\mathbb{E}Z^k$ for all $1\leq i\leq n_j$, $1\leq j\leq d$ and all positive integers $k\leq p$.  Suppose also that $\mathbb{E}|X_{ij}|^{r_l+p+2}<\infty$ for all $i$, $j$ and $l$.  Then, for $h\in C_b^{p+2}(\mathbb{R})$, 
\begin{align*}|\mathbb{E}h(g(\mathbf{W}))-&\mathbb{E}h(g(\mathbf{Z}))| 
\leq\frac{1}{p!}h_{p+2}\bigg\{\frac{1}{p+2}\sum_{j=1}^d\sum_{i=1}^{n_j}\frac{1}{n_j^{p/2+1}}\bigg(\frac{p+2}{p+1}+|\mathbb{E}X_{ij}^{p+1}|\bigg)\bigg[A\mathbb{E}|X_{ij}|^{p+2} \\
&\quad+B\sum_{k=1}^d2^{r_k}\bigg(2^{r_k}\mathbb{E}|X_{ij}|^{p+2}\mathbb{E}|W_k|^{r_k}+\frac{2^{r_k}}{n_k^{r_k/2}}\mathbb{E}|X_{ij}^{p+2}X_{ik}^{r_k}|+\mathbb{E}|Z|^{r_k}\mathbb{E}|X_{ij}|^{p+2}\bigg)\bigg]\\
&\quad  +\frac{3}{2}\sum_{j=1}^d\sum_{i=1}^{n_j}\frac{|\mathbb{E}X_{ij}^{p+1}|}{n_j^{(p+1)/2}}\sum_{k=1}^d\sum_{l=1}^{n_k}\frac{1}{n_k^{3/2}}\bigg[A\mathbb{E}|X_{lk}|^3 \\
&\quad+B\sum_{t=1}^d\!3^{r_t}\bigg(\!2^{r_t}\mathbb{E}|X_{lk}|^3\mathbb{E}|W_t|^{r_t}\!+\!\frac{2^{r_t}}{n_t^{r_t/2}}\mathbb{E}|X_{lk}^3X_{lt}^{r_t}|\!+\!2\mathbb{E}|Z|^{r_t+1}\mathbb{E}|X_{lk}|^{3}\!\bigg)\bigg]\bigg\}.
\end{align*}
\end{theorem}

\begin{theorem}\label{multievenguni} Let $P(w)=A+B|w|^r$, where $A$, $B$ and $r$ are non-negative constants.  Suppose $g\in C_P^{p}(\mathbb{R})$ is an even function. Let $X_{1},\ldots,X_{n}$ be independent random variables with $\mathbb{E}X_{i}^k=\mathbb{E}Z^k$ for all $1\leq i\leq n$ and all positive integers $k\leq p$.  Suppose also that $\mathbb{E}|X_{i}|^{r+p+4}<\infty$ for all $1\leq i\leq n$.  Then, for $h\in C_b^{p}(\mathbb{R})$,
\begin{align*}&|\mathbb{E}h(g(W))-\mathbb{E}h(g(Z))| \leq \frac{1}{p!n^{p/2+1}}h_p\bigg\{\frac{1}{p+2}\sum_{i=1}^{n}\bigg(\frac{p+2}{p+1}+|\mathbb{E}X_{i}^{p+1}|\bigg)\\
&\quad\times\bigg[3A\mathbb{E}|X_{i}|^{p+2}+2^rB\bigg(2^{r+1}\mathbb{E}|X_i|^{p+2}\big(\mathbb{E}|W|^{r+1}+\mathbb{E}|W|^{r}\big) \\
&\quad +4\mathbb{E}|Z|^{r+1}\mathbb{E}|X_i|^{p+3}\!+\!\frac{2^{r+2}}{n^{r/2}}\mathbb{E}|X_i|^{r+p+3}\!\bigg)\bigg]\!+\!\frac{3}{2n}\sum_{i=1}^{n}\sum_{l=1}^{n}\!|\mathbb{E}X_{i}^{p+1}|\bigg[10A\mathbb{E}|X_l|^{p+2} \\
&\quad+3^{r+1}B\bigg(2^{r+1}\mathbb{E}|X_l|^{p+2}\big(2\mathbb{E}|W|^{r+2}+\mathbb{E}|W|^{r}\big)+16\mathbb{E}|Z|^{r+1}\mathbb{E}|X_l|^{p+4}+\frac{2^{r+3}}{n^{r/2}}\mathbb{E}|X_l|^{r+p+4}\bigg)\bigg]\bigg\}.
\end{align*}
\end{theorem}

\begin{remark}\label{wass}Consider Theorem \ref{winfirst2} with $p=2$.  Then setting $h_1=\|h'\|=1$ gives a bound in the Wasserstein metric. All other bounds given in Theorems \ref{winfirst1} -- \ref{multievenguni} can only be given in weaker metrics, though.  Consider now Theorem \ref{winfirst1} with $p\geq2$, and set $\|h^{(k)}\|=1$ for all $k=1,\ldots,p$.  Then $h_p=\sum_{k=1}^p{p\brace k}=B_p$, where $B_p=\mathrm{e}^{-1}\sum_{j=1}^\infty\frac{j^p}{j!}$ is the $p$-th Bell number (see Section 26.7(i) of \cite{olver}, noting that ${n\brace 0}=0$ for $n\geq1$).  Thus, setting $h_p=B_p$ in Theorem \ref{winfirst2} gives a bound in the smooth Wasserstein metric $d_{\mathcal{H}_p}$, $p\geq2$.  Similarly, the bounds from Theorems \ref{winfirst2} -- \ref{multievenguni} can be given in smooth Wasserstein metrics.
\end{remark}

\begin{remark}From Theorem \ref{winfirst1}, we see that if the first $p$ moments of the $X_{ij}$ match those of the $N(0,1)$ distribution then, provided $\mathbb{E}|X_{ij}|^{r_l+p+1}<\infty$ for all $i,j,l$ and $g\in C_{P_1}^p(\mathbb{R}^d)$ for $P_1(\mathbf{w})=A+B\sum_{i=1}^d|w_i|^{r_i}$, the rate of convergence of $g(\mathbf{W})$ to $g(\mathbf{Z})$ is $O(n^{-(p-1)/2})$, where $n=\min_{1\leq j\leq d}n_j$.  For this rate of convergence, we require that the test function $h$ is in the class $C_b^p(\mathbb{R})$.  By Theorem \ref{multieveng}, it follows that, for even $p$, the rate of convergence can be improved to $O(n^{-p/2})$ if we strengthen the assumptions to $\mathbb{E}|X_{ij}|^{r_l+p+2}<\infty$ for all $i,j,l$, that $g_{P_1}^{p+2}(\mathbb{R}^d)$ is an even function, and that $h\in C_b^{p+2}(\mathbb{R})$.  When $d=1$, we see from Theorems \ref{winfirst2} and \ref{multievenguni} that we can achieve these convergence rates with weaker assumptions on $g$ and $h$, possibly at the expense of stronger conditions on the existence of absolute moments of the $X_i$.
\end{remark}

\begin{remark}\label{rem3.6}One could derive analogues of Theorems \ref{winfirst1} -- \ref{multievenguni} for the case that the dominating function of $g$ is of the form $P_2(\mathbf{w})=A\exp(a\sum_{i=1}^d|w_i|^b)$, where $0<b\leq2$.  Such results could be derived by using the bounds of Section \ref{sec2point3} to obtain an analogue of Lemma \ref{cbwbshc} and substitute the resulting bounds into the general bounds of Lemmas \ref{noteveng} and \ref{evenglemmabound}.  In this paper, we do not carry out the more tedious and involved calculations required to obtain such analogues of Theorems \ref{winfirst1} -- \ref{multievenguni}, but we note here that it can easily be seen that similar principles regarding the rate of convergence of $g(\mathbf{W})$ to $g(\mathbf{Z})$ would apply.  For example, the analogue of Theorem \ref{winfirst1} would give a bound on $|\mathbb{E}h(g(\mathbf{W}))-\mathbb{E}h(g(\mathbf{Z}))|$ that was $O(n^{-(p-1)/2})$ under the following assumptions.  The first $p$ moments of the $X_{ij}$ would also agree with those of the $N(0,1)$ distribution, although the absolute moment assumption would be replaced by the condition that an expectation of the form $\mathbb{E}[|X_{ij}|^\alpha\exp(\beta |X_{ij}|^\gamma)]$ would exist, for some $\alpha,\beta,\gamma>0$.  The assumptions on $g$ and $h$ would be exactly analogous with $g\in C_{P_2}^p(\mathbb{R}^d)$ and $h\in C_b^p(\mathbb{R})$.  
\end{remark}

The rates of convergence of Theorems \ref{winfirst1} -- \ref{multievenguni} cannot be improved.

\begin{proposition}\label{prop3.5}(i). For any $p\geq2$, the $O(n^{-(p-1)/2})$ rate of Theorems \ref{winfirst1} and \ref{winfirst2} cannot be improved.

(ii). For any even $p\geq2$, the $O(n^{-p/2})$ rate of Theorems \ref{multieveng} and \ref{multievenguni} cannot be improved.
\end{proposition}

The proof shall use the following lemma.

\begin{lemma}\label{lem3.6}Let $X,X_1,\ldots,X_n$ be i.i.d$.$ random variables, with $\mathbb{E}X^k=\mathbb{E}Z^k$ for $k=1,\ldots,p$. 

(i). If $\mathbb{E}|X|^{p+1}<\infty$, then
\begin{equation*}\label{ordercor1}\mathbb{E}W^{p+1}=\mathbb{E}Z^{p+1}+\frac{1}{n^{(p-1)/2}}\big(\mathbb{E}X^{p+1}-\mathbb{E}Z^{p+1}\big).
\end{equation*}

(ii). If we further assume that $\mathbb{E}|X|^{p+2}<\infty$, then
\begin{equation}\label{ordercor2}\mathbb{E}W^{p+2}=\mathbb{E}Z^{p+2}+\frac{1}{n^{p/2}}\big(\mathbb{E}X^{p+2}-\mathbb{E}Z^{p+2}\big).
\end{equation}
\end{lemma}

\begin{proof}  Let us prove part (ii); the proof of part (i) is similar and slightly easier.  Since
\begin{equation*}\mathbb{E}W^{p+2}=\frac{1}{n^{(p+2)/2}}\sum_{i_1=1}^n\cdots\sum_{i_{p+2}=1}^n\mathbb{E}[X_{i_1}\cdots X_{i_{p+2}}]
\end{equation*}
and the random variables $X_1,\ldots,X_n$ are i.i.d$.$, it follows that $\mathbb{E}W^{p+2}$ is expressed solely as a sum of constants and products of moments $\mathbb{E}X^k$, $k=1,\ldots,p+2$.  In fact, only the moment $\mathbb{E}X^{p+2}$ appears in the sum.  This is because the first $p$ moments are equal to the moments of the standard normal distribution and thus become constants.  Also, the moment $\mathbb{E}X^{p+1}$ only appears in the sum as a product with the moment $\mathbb{E}X=0$, and thus vanishes.  Therefore, $\mathbb{E}W^{p+2}=a+b\mathbb{E}X^{p+2}$ for some constants $a$ and $b$.  There are only $n$ terms in the sum in which $\mathbb{E}X^{p+2}$ is present, and so $b=n^{-p/2}$.  Also, if $\mathbb{E}X^{p+2}=\mathbb{E}Z^{p+2}$ then we must have that $\mathbb{E}W^{p+2}=\mathbb{E}Z^{p+2}$, from which we deduce that $a=(1-n^{-p/2})\mathbb{E}Z^{p+2}$.  We thus obtain (\ref{ordercor2}), as required.
\end{proof}

\noindent{\emph{Proof of Proposition \ref{prop3.5}.} We first prove that the  $O(n^{-(p-1)/2})$ rate of Theorem \ref{winfirst2} cannot be improved, which also shows that the $O(n^{-(p-1)/2})$ rate of Theorem \ref{winfirst1} cannot be improved.  Let $X,X_1,\ldots,X_n$ be i.i.d$.$ random variables satisfying the assumptions of Theorem \ref{winfirst2}, so that $\mathbb{E}X^k=\mathbb{E}Z^k$ for all positive integers $k\leq p$, and $\mathbb{E}|X|^{p+1}<\infty$.  Suppose, however, that $\mathbb{E}X^{p+1}\not=\mathbb{E}Z^{p+1}$.   Take $h(w)=w$, which is in the class $C_b^{p-1}(\mathbb{R})$, and $g_1(w)=w^{p+1}$, whose derivatives have polynomial growth.  Then, by part (i) of Lemma \ref{lem3.6}, 
\begin{equation*}\mathbb{E}h(g_1(W))-\mathbb{E}h(g_1(Z))=\mathbb{E}W^{p+1}-\mathbb{E}Z^{p+1}=\frac{1}{n^{(p-1)/2}}\big(\mathbb{E}X^{p+1}-\mathbb{E}Z^{p+1}\big),
\end{equation*} 
and so the $O(n^{-(p-1)/2})$ rate cannot be improved.

The proof of part (ii) is similar.  Suppose that $X,X_1,\ldots,X_n$ are as before, but with the additional assumption that $\mathbb{E}X^{p+2}$, where $p$ is even.  Take $h(w)=w$ and $g_2(w)=w^{p+2}$, which is an even function.  Then, by part (ii) of Lemma \ref{lem3.6}, 
\begin{equation*}\mathbb{E}h(g_2(W))-\mathbb{E}h(g_2(Z))=\mathbb{E}W^{p+2}-\mathbb{E}Z^{p+2}=\frac{1}{n^{p/2}}\big(\mathbb{E}X^{p+2}-\mathbb{E}Z^{p+2}\big),
\end{equation*} 
and so the $O(n^{-p/2})$ rate cannot be improved. \hfill $\Box$

\subsection{Examples}\label{secex}

We end with some simple examples that are chosen to illuminate the results of this section.  For simplicity, we consider the case that the $X_{ij}$ are i.i.d$.$ and are equal in law to the random variable $X$.

\subsubsection{Normal approximation: $g(x)=x$}

In the classical case $g(x)=x$, we can apply Theorem \ref{winfirst2} with $A=1$ and $B=0$ to obtain the bound
\begin{equation*}|\mathbb{E}h(W)-\mathbb{E}h(Z)|\leq \frac{3(p+1)}{p!n^{(p-1)/2}}h_{p-1}\mathbb{E}|X|^{p+1}.
\end{equation*}
The case $p=3$ (the first three moments match those of the standard normal) can be compared with Corollary 3.1 of \cite{goldstein}, in which a bound of order $n^{-1}$ was obtained using zero bias couplings.  The requirement that $\mathbb{E}X^4<\infty$ is common to both bounds, but our bound only requires that $h\in C_b^2(\mathbb{R})$, whereas \cite{goldstein} require that $h\in C_b^4(\mathbb{R})$.  For the case of general $p$, we can compare to Corollary 3.2 of \cite{gaunt rate}.  Whilst the requirements on the moments and the test function $h$ are the same, our bound is outperformed by that of \cite{gaunt rate}, which one may expect because our approach is much more general.
 
\subsubsection{Chi-square approximation: $g(x)=x^2$}\label{chise}

Let us now study the approximation of $W^2$ by $Z^2\sim\chi_{(1)}^2$.  Firstly, we apply Theorem \ref{winfirst2}.  Here $g'(w)=2w$, so we take $P(w)=2w$ as our dominating function, and applying the theorem with $p=2$, $A=0$, $B=2$ and $r=1$.  Using that $\mathbb{E}|W|\leq(\mathbb{E}W^2)^{1/2}=1$, $\mathbb{E}|Z|^3=2\sqrt{2/\pi}$ and recalling Remark \ref{wass} we obtain the bound
\begin{equation}\label{h5}d_{\mathrm{W}}(\mathcal{L}(W^2),\chi_{(1)}^2)\leq \frac{48}{\sqrt{n}}\bigg[\mathbb{E}|X|^3+\sqrt{\frac{2}{\pi}}\mathbb{E}X^4+\frac{\mathbb{E}|X|^5}{\sqrt{n}}\bigg].
\end{equation}
It is worth noting that obtaining $O(n^{-1/2})$ Wasserstein distance bounds using the chi-square Stein equation would present technical challenges, and this is the first instance of such a bound in the Stein's method literature.

Since $g(w)=w^2$ is an even function with polynomial growth, the $O(n^{-1/2})$ rate can be improved, at the expense of working in a weaker probability metric.  Let us apply Theorem \ref{multievenguni}.  We take $P(w)=2+4w^2$ as our dominating function, since $|g'(w)|^2=4w^2$ and $g''(w)=2$ are both bounded by $P(w)$ for all $w\in\mathbb{R}$.  We then apply the theorem with $p=2$, $A=2$, $B=4$ and $r=2$.  To obtain a more compact form for the final bound, we simplify by using that $1\leq\mathbb{E}|X|^a\leq\mathbb{E}|X|^b$ and $1\leq\mathbb{E}|W|^a\leq\mathbb{E}|W|^b$, for $2\leq a\leq b$, which follows from H\"{o}lder's inequality.  We also round numbers up to the nearest integer to obtain
\begin{align}\label{h6}|\mathbb{E}h(W^2)-\chi_{(1)}^2h|&\leq\frac{1}{n}\big(\|h''\|+\|h'\|\big)\bigg[22\mathbb{E}|W|^3+40\mathbb{E}|X|^5+\frac{43}{n}\mathbb{E}|X|^7\nonumber\\
&\quad+|\mathbb{E}X^3|\bigg(1312\mathbb{E}X^4\mathbb{E}W^4+3974\mathbb{E}X^6+\frac{2592}{n}\mathbb{E}X^8\bigg)\bigg],
\end{align} 
where $\chi_{(1)}^2h$ denotes the expectation of $h(U)$ for $U\sim\chi_{(1)}^2$.  It is interesting to compare this bound with the previous example; the simple act of taking the square of the random variable means that we can obtain a $O(n^{-1})$ bound even without the first three moments matching those of the standard normal distribution.  

A bound for the quantity $|\mathbb{E}h(W^2)-\chi_{(1)}^2|$ has also been obtained by \cite{gaunt chi square}, through a different approach involving the chi-square Stein equation.  They also required that $\mathbb{E}X^{8}<\infty$; however, their bound requires that $h\in C_b^3(\mathbb{R}_+)$.  Finally, we note that we could obtain a $O(n^{-1})$ bound for the rate of convergence of $\sum_{k=1}^d W_k^2$ to the $\chi_{(d)}^2$ distribution using Theorem \ref{multieveng}.  However, the dependence on $d$ would be bad; for large $d$ the order would be $d^3n^{-1}$.  Instead, it is better to use the bounds (\ref{h5}) of (\ref{h6}) and the simple conditioning argument used to prove Theorem 3.3 of \cite{gaunt chi square} to obtain a $O(d n^{-1})$ bound.  This rate is worse than the $n^{-1}$ rate of Theorem 3.3 of \cite{gaunt chi square}.

\subsubsection{Hermite polynomials: $g(x)=H_n(x)$}

Since the seminal paper \cite{np09}, there has been considerable interest in the link between Stein's method and Malliavin calculus; see the book \cite{np12} for a detailed account of normal approximation by the Malliavin-Stein method. As noted by \cite{peccati14}, an important class of limits which fall outside current the state of the art of the Malliavin-Stein method are those of the type $Q(Z)$, where $Z\sim N(0,1)$ and $Q$ is polynomial of degree strictly greater than 2.  In particular, the case that $P$ is a Hermite polynomial, as given by $H_q(x)=(-1)^q\mathrm{e}^{x^2/2}\frac{\mathrm{d}^q}{\mathrm{d}x^q}(\mathrm{e}^{-x^2/2})$, $q\geq1$, is of particular interest, due to their fundamental role in Gaussian analysis and Malliavin calculus. 

The difficulty in applying the Malliavin-Stein method here is because the only Stein equations in the current literature for $H_q(Z)$, $q\geq3$, are those of \cite{gaunt 34} for $H_3(Z)$ and $H_4(Z)$, which are fifth and third order differential equations, respectively, for which no estimates are given for the solution of the Stein equation.  However, the limit distributions $H_q(Z)$, $q\geq1$, evidently fall into our framework with $g(x)=H_q(x)$. We note that $H_q(x)$ is a polynomial of degree $q$ and if $q$ is odd (even) then $H_q(x)$ is an odd (even) function of $x$.  Therefore, as we did in Section \ref{chise}, we can apply Theorem \ref{winfirst2} (with $P(w)=C_q(1+|w|^{q-1})$ for some absolute constant $C_q$) and Theorem \ref{multievenguni} (with $P(w)=C_q'(1+|w|^{2q-2})$) to obtain the bounds
\begin{equation*}d_{\mathrm{W}}(\mathcal{L}(H_q(W)),\mathcal{L}(H_q(Z)))\leq\frac{K_q\mathbb{E}|X|^{q+3}}{\sqrt{n}}, \quad q\geq3
\end{equation*}
and
\begin{equation*}|\mathbb{E}h(H_{2q}(W))-\mathbb{E}h(H_{2q}(Z))|\leq\frac{M_q}{n}(\|h'\|+\|h''\|)(1+|\mathbb{E}X^3|)\mathbb{E}|X|^{2q+4}, \quad q\geq2, 
\end{equation*}
where $K_q$ and $M_q$ are absolute constants, which can be found from a careful bookkeeping of constants.  In obtaining these compact bounds we used that, for $p\geq2$, there exists a constant $B_p>0$ such that $\mathbb{E}|W|^{p}\leq B_p\mathbb{E}|X|^p$, which follows from the Marcinkiewicz–Zygmund inequality.

This example highlights some of the strengths and weaknesses of the theory developed in this paper.  On the one hand, we are able to easily give quantitative limit theorems for limit distributions that cannot otherwise be dealt with by Stein's method, but, on the other hand, the class of prelimits we can treat is very restrictive and our theory cannot treat limit theorems involving convergence to linear combinations of Hermite polynomials that are found in applications in works such as \cite{tudor}.

\vspace{3mm}

\subsubsection{Approximation of expectations of smooth functions of binomial and Poisson random variables}

Let $S\sim \mathrm{Bin}(n,p)$.  Then it is a well-known and standard application of the central limit
 theorem that $W=\frac{S-np}{\sqrt{np(1-p)}}$ converges in distribution to the standard normal distribution.  Letting $Y_1,\ldots,Y_n\sim \mathrm{Ber}(p)$ be i.i.d$.$ Bernoulli random variables, we can write $S=\sum_{i=1}^nY_i$.  We can thus write $W=\frac{1}{\sqrt{n}}\sum_{i=1}^nX_i,$ where $X_i=\frac{Y_i-p}{\sqrt{p(1-p)}}$.  All moments of the $X_i$ exist, so if we take $h(w)=w$ and $g:\mathbb{R}\rightarrow\mathbb{R}$ has a first derivative with polynomial growth, then we can apply Theorem \ref{winfirst2} to obtain the bound 
\begin{equation}\label{tyui}|\mathbb{E}g(W)-\mathbb{E}g(Z)|\leq \frac{C_{p,g}}{\sqrt{n}},
\end{equation}
where $C_{p,g}$ depends on $p$ and $g$, but not $n$.  If we assume further that $g$ has a second derivative with polynomial growth and that either $p=\frac{1}{2}$ ($\mathbb{E}X_i^3=0$ if and only if $p=\frac{1}{2}$) or $g$ is an even function, then we can apply Theorem \ref{multievenguni} to obtain the improved convergence rate
\begin{equation}\label{tyu}|\mathbb{E}g(W)-\mathbb{E}g(Z)|\leq \frac{K_{p,g}}{n}.
\end{equation}


The Poisson random variable $T\sim \mathrm{Po}(\lambda)$, can be decomposed similarly to the binomial distribution: $T$ is equal in distribution to $\sum_{i=1}^{\lfloor\lambda\rfloor}V_i$, where the $V_i$ are independent $\mathrm{Po}\big(\frac{\lambda}{\lfloor\lambda\rfloor}\big)$ random variables and the floor function $\lfloor x\rfloor$ is the greatest integer less than or equal to $x$.  It is evident that analogous bounds to (\ref{tyui}) and (\ref{tyu}) can be obtained for the normalised Poisson random variable $\frac{T-\lambda}{\sqrt{\lambda}}$.  The only difference is that here $\mathbb{E}X_i^3$ is not equal to zero for any value of $\lambda$, so to obtain the $\lambda^{-1}$ rate we always require that $g$ is an even function.  We refer the reader to \cite{gaunt cmp} for further details, in which such estimates find a surprising application in a derivation of the leading term in the asymptotic expansion of the normalising constant of the Conway-Maxwell-Poisson distribution.


\subsubsection{The delta method}\label{deltasec}

Let $X_1,\ldots,X_n$ be independent random variables with zero mean and unit variance, and let $\overline{X}$ denote the sample mean.  Then, by the delta method,
\begin{equation}\label{deltamethod}\sqrt{n}\big(f(\overline{X})-f(0)\big)\stackrel{\mathcal{D}}{\rightarrow} N(0,[f'(0)]^2)
\end{equation}
for continuous functions $f$ with $f'(0)\not=0$. In Theorem \ref{thmdelta1} below we obtain a $O(n^{-1/2})$ bound for the rate of convergence for the case that the derivatives of $f$ have polynomial growth by applying Theorem \ref{winfirst2} with $g(w)=\sqrt{n}\big(f(w/\sqrt{n})-f(0)\big)/f'(0)$ (note that $W=\sqrt{n}\overline{X}$).  
In the degenerate case that $f'(0)=0$ but $f''(0)\not=0$, we have instead that
\begin{equation}\label{deltamethod2}n\big(f(\overline{X})-f(0)\big)\stackrel{\mathcal{D}}{\rightarrow}\frac{1}{2}f''(0)\chi_{(1)}^2,
\end{equation}
where $\chi_{(1)}^2$ is the chi-square distribution with 1 degree of freedom (see \cite[Theorem 5.5]{h14}).  We can again obtain a $O(n^{-1/2})$ rate of convergence by applying Theorem \ref{winfirst2}, and in Theorem \ref{thmdelta2} below we use Theorem \ref{multievenguni} to obtain a bound with $O(n^{-1})$ rate of convergence for the case that $f$ is an even function (in which case $f'(0)=0$).  As far as the author is aware, this is the first bound in the literature with a $O(n^{-1})$ rate of convergence for the delta method and certainly the first to have been proved using Stein's method.

The limiting results (\ref{deltamethod}) and (\ref{deltamethod2}) generalise to functions $f:\mathbb{R}^d\rightarrow\mathbb{R}^m$.  If the partial derivatives of $f$ are of polynomial growth, we can use Theorems \ref{winfirst1} and (\ref{multieveng}) to obtain bounds for $m=1$, but none of our bounds are applicable for the more general case.  We also refer the reader to \cite{pm16} for a detailed investigation into convergence rates in the multivariate delta method, with optimal order bounds given in the Kolmogorov metric.

\begin{theorem}\label{thmdelta1}Suppose that $f:\mathbb{R}\rightarrow\mathbb{R}$ is twice differentiable with $f'(0)\not=0$, $|f'(w)|\leq A+B|w|^r$ and $|f''(w)|\leq C+D|w|^s$ for all $w\in\mathbb{R}$, where $A$, $B$, $C$, $D$, $r$ and $s$ are non-negative constants.  Let $X_{1},\ldots,X_{n}$ be independent random variables with zero mean, unit variance and $\mathbb{E}|X_i|^{r+4}<\infty$ for all $1\leq i\leq n$.  Define $T_1=\sqrt{n}\big(f(\overline{X})-f(0)\big)/f'(0)$. Then, for absolutely continuous $h:\mathbb{R}\rightarrow\mathbb{R}$, 
\begin{align*}|\mathbb{E}h(T_1)-\mathbb{E}h(Z)|
&\leq\frac{\|h'\|}{|f'(0)|\sqrt{n}}\bigg\{\frac{1}{n}\sum_{i=1}^{n}\bigg[3A\mathbb{E}|X_{i}|^{3}+\frac{2^rB}{n^{r/2}}\bigg(2^{r+1}\mathbb{E}X_i^{4}\big(\mathbb{E}|W|^{r+1}\!+\!\mathbb{E}|W|^{r}\big)\\
&\quad+4\mathbb{E}|Z|^{r+1}\mathbb{E}X_i^{4}+\frac{2^{r+2}}{n^{r/2}}\mathbb{E}|X_i|^{r+4}\bigg)\bigg]+\bigg(C+\frac{D}{n^{s/2}}\mathbb{E}|Z|^{s+2}\bigg)\bigg\}.
\end{align*}
\end{theorem}

\begin{theorem}\label{thmdelta2} Let $f:\mathbb{R}\rightarrow\mathbb{R}$ be twice differentiable with $f'(0)=0$ and $f''(0)\not=0$.  Let $X_{1},\ldots,X_{n}$ be independent random variables with zero mean and unit variance.  Define $T_2=2n\big(f(\overline{X})-f(0)\big)/f''(0)$ and let $Y\sim\chi_{(1)}^2$.

\vspace{3mm}

\noindent (i)  In addition to the above, suppose that $f\in C^3(\mathbb{R})$ with $|f''(w)|\leq A+B|w|^r$ and $|f^{(3)}(w)|\leq C+D|w|^s$ for all $w\in\mathbb{R}$, where $A$, $B$, $C$, $D$, $r$ and $s$ are non-negative constants, and also that $\mathbb{E}|X_i|^{r+4}<\infty$, $1\leq i\leq n$.  Then, for absolutely continuous $h:\mathbb{R}\rightarrow\mathbb{R}$, 
\begin{align*}|\mathbb{E}h(T_2)-\mathbb{E}h(Y)|
&\leq\frac{\|h'\|}{|f''(0)|\sqrt{n}}\bigg\{\frac{1}{n}\sum_{i=1}^{n}\bigg[6A\mathbb{E}|X_{i}|^{3}+\frac{2^{r+1}B}{n^{r/2}}\bigg(2^{r+1}\mathbb{E}X_i^{4}\big(\mathbb{E}|W|^{r+1}\!+\!\mathbb{E}|W|^{r}\big)\\
&\quad+4\mathbb{E}|Z|^{r+1}\mathbb{E}X_i^{4}+\frac{2^{r+2}}{n^{r/2}}\mathbb{E}|X_i|^{r+4}\bigg)\bigg]+\frac{1}{3}\bigg(2\sqrt{\frac{2}{\pi}}C+\frac{D}{n^{s/2}}\mathbb{E}|Z|^{s+3}\bigg)\bigg\}.
\end{align*}

\noindent (ii) Suppose now that $f\in C^4(\mathbb{R})$ is an even function with $|f''(w)|\leq A+B|w|^r$ and $|f^{(4)}(w)|\leq C+D|w|^s$ for all $w\in\mathbb{R}$, where $A$, $B$, $C$, $D$, $r$ and $s$ are non-negative constants, and also that $\mathbb{E}|X_i|^{r+6}<\infty$, $1\leq i\leq n$.  Then, for $h\in C_b^{2}(\mathbb{R})$,
\begin{align*}&|\mathbb{E}h(T_2)-\mathbb{E}h(Y)| \leq \frac{1}{|f''(0)|n}\bigg\{\big(\|h''\|+\|h'\|\big)\bigg\{\frac{1}{n}\sum_{i=1}^{n}\bigg(\frac{1}{3}+\frac{|\mathbb{E}X_{i}^{3}|}{4}\bigg)\\
&\quad\times\bigg[3A\mathbb{E}X_{i}^{4}+\frac{2^rB}{n^{r/2}}\bigg(2^{r+1}\mathbb{E}X_i^{4}\big(\mathbb{E}|W|^{r+1}+\mathbb{E}|W|^{r}\big) \\
&\quad +4\mathbb{E}|Z|^{r+1}\mathbb{E}|X_i|^{5}+\frac{2^{r+2}}{n^{r/2}}\mathbb{E}|X_i|^{r+5}\!\bigg)\bigg]\!+\!\frac{3}{2n^2}\sum_{i=1}^{n}\sum_{l=1}^{n}|\mathbb{E}X_{i}^{3}|\bigg[10A\mathbb{E}X_l^{4} \\
&\quad+\frac{3^{r+1}B}{n^{r/2}}\bigg(2^{r+1}\mathbb{E}X_l^{4}\big(2\mathbb{E}|W|^{r+2}+\mathbb{E}|W|^{r}\big)+16\mathbb{E}|Z|^{r+1}\mathbb{E}X_l^{6}+\frac{2^{r+3}}{n^{r/2}}\mathbb{E}|X_l|^{r+6}\bigg)\bigg]\bigg\} \\
&\quad +\frac{\|h'\|}{12}\bigg(3C+\frac{D}{n^{s/2}}\mathbb{E}|Z|^{s+4}\bigg) +\frac{\|h''\|}{|f''(0)|}\bigg[\frac{5}{3}(f^{(3)}(0))^2+\frac{1}{144n}\big(105C^2+\frac{D^2}{n^{s/2}}\mathbb{E}|Z|^{s+8}\bigg)\bigg]\bigg\}.
\end{align*}
\end{theorem}

\vspace{3mm}

\noindent \emph{Proof of Theorem \ref{thmdelta1}.}  Let $g_1(w)=\sqrt{n}\big(f(w/\sqrt{n})-f(0)\big)/f'(0)$ and define $W=\sqrt{n}\overline{X}$, so that $W=\frac{1}{\sqrt{n}}\sum_{i=1}^nX_i$.  Note that $T_1=g_1(W)$. Then, by the triangle inequality,
\begin{align*}|\mathbb{E}h(T_1)-\mathbb{E}h(Z)|&\leq |\mathbb{E}h(g_1(W))-\mathbb{E}h(g_1(Z))|+|\mathbb{E}h(g_1(Z))-\mathbb{E}h(Z))|\\
&=:R_1+R_2.
\end{align*}
We now note that $g_1'(w)=f'(w/\sqrt{n})/f'(0)$, and therefore $|g_1'(w)|\leq (A+Bn^{-r/2}|w|^r)/|f'(0)|$ for all $w\in\mathbb{R}$.  Hence we can apply Theorem \ref{winfirst2} with the dominating function $P(w)=(A+Bn^{-r/2}|w|^r)/|f'(0)|$ to bound $R_1$.  We now bound $R_2$.  By Taylor expanding $f(w/\sqrt{n})$ about $0$, we have that
\begin{align*}\mathbb{E}h(g_1(Z))=\mathbb{E}h\Big(Z+\frac{f''(\theta Z/\sqrt{n})}{f'(0)\sqrt{n}}Z^2\Big),
\end{align*}  
for some $\theta\in(0,1)$.  Therefore, as $|f''(w)|\leq C+D|w|^s$ for all $w\in\mathbb{R}$,
\begin{align*}|R_2|=|\mathbb{E}h(g_1(Z))-\mathbb{E}h(Z))|&\leq \frac{\|h'\|}{|f'(0)|\sqrt{n}}\mathbb{E}|Z^2f''(\theta Z/\sqrt{n})|\\
&\leq \frac{\|h'\|}{|f'(0)|\sqrt{n}}\mathbb{E}\bigg[Z^2\bigg(C+D\bigg|\frac{\theta Z}{\sqrt{n}}\bigg|^s\bigg)\bigg]\leq \frac{\|h'\|}{|f'(0)|\sqrt{n}}\bigg(C+\frac{D}{n^{s/2}}\mathbb{E}|Z|^{s+2}\bigg).
\end{align*}
Combining our bounds for $R_1$ and $R_2$ yields the desired bound and completes the proof. \hfill $\Box$

\vspace{3mm} 

\noindent\emph{Proof of Theorem \ref{thmdelta2}.} (i) Let $g_2(w)=2n\big(f(w/\sqrt{n})-f(0)\big)/f''(0)$.  Let $W=\sqrt{n}\overline{X}$ so that $T_2=g_2(W)$.  Then, by the triangle inequality,
\begin{align*}|\mathbb{E}h(T_2)-\mathbb{E}h(Z)|&\leq |\mathbb{E}h(g_2(W))-\mathbb{E}h(g_2(Z))|+|\mathbb{E}h(g_2(Z))-\mathbb{E}h(Y))|\\
&=:R_1+R_2.
\end{align*}
On Taylor expanding $g_2'(w)=2\sqrt{n}f'(w/\sqrt{n})/f''(0)$ about 0 and using that $f'(0)=0$ and $|f''(w)|\leq A+B|w|^r$ for all $w\in\mathbb{R}$, we have that $|g_2'(w)|\leq 2(A+B|w/\sqrt{n}|^r)/|f''(0)|$ for all $w\in\mathbb{R}$.  Hence we can bound $R_1$ by applying Theorem \ref{winfirst2} with dominating function $P(w)=2(A+Bn^{-r/2}|w|^r)/|f''(0)|$.  Let us now bound $R_2$. By Taylor expanding $f(w/\sqrt{n})$ about $0$ and using that $f'(0)=0$, we have that 
\begin{align*}\mathbb{E}h(g_2(Z))=\mathbb{E}h\Big(Z^2+\frac{f^{(3)}(\theta Z/\sqrt{n})}{3f''(0)\sqrt{n}}Z^3\Big),
\end{align*}
for some $\theta\in(0,1)$.  Therefore, as $V\stackrel{\mathcal{D}}{=}Z^2$ and $|f^{(3)}(w)|\leq C+D|w|^s$ for all $w\in\mathbb{R}$,
\begin{align*}|R_2|=|\mathbb{E}h(g_2(Z))&-\mathbb{E}h(Z^2))|\leq \frac{\|h'\|}{3|f''(0)|\sqrt{n}}\mathbb{E}|Z^3f^{(3)}(\theta Z/\sqrt{n})|\\
&\leq \frac{\|h'\|}{3|f''(0)|\sqrt{n}}\mathbb{E}\bigg|Z^3\bigg(C+D\bigg|\frac{\theta Z}{\sqrt{n}}\bigg|^s\bigg)\bigg|\leq \frac{\|h'\|}{3|f''(0)|\sqrt{n}}\bigg(2\sqrt{\frac{2}{\pi}}C+\frac{D}{n^{s/2}}\mathbb{E}|Z|^{s+3}\bigg), 
\end{align*}
where we used that $\mathbb{E}|Z|^3=2\sqrt{2/\pi}$ in the final step. Combining the bounds for $R_1$ and $R_2$ now yields the desired bound.

\vspace{3mm}

\noindent (ii) Suppose now that $f$ is an even function, so that $g_2$ is also an even function. 
We have already shown in part (i) that $|g_2'(w)|\leq 2(A+B|w/\sqrt{n}|^r)^{1/2}/|f''(0)|$ for all $w\in\mathbb{R}$.  We also have that $|g_2''(w)|=2|f''(w/\sqrt{n})|/|f''(0)|\leq 2(A+B|w/\sqrt{n}|^r)/|f''(0)|$. Therefore, we can obtain an alternative bound for $R_1$ to the one given in part (i) of the proof by applying Theorem \ref{multievenguni} with dominating function $P(w)=2(A+Bn^{-r/2}|w|^r)/|f''(0)|$.  To bound $R_2$, we begin by proceeding as we did in part (i) but this time Taylor expand one term further to obtain
\begin{align*}\mathbb{E}h(g_2(Z))=\mathbb{E}h\Big(Z^2+\frac{f^{(3)}(0)}{3f''(0)\sqrt{n}}Z^3+\frac{f^{(4)}(\theta Z/\sqrt{n})}{12f''(0)n}Z^4\Big),
\end{align*}
for some $\theta\in(0,1)$.  By another Taylor expansion, we then have that
\begin{align*}|R_2|=|\mathbb{E}h(g_2(Z))-\mathbb{E}h(Z)|&=|R_3+R_4|,
\end{align*}
where
\begin{align*}R_3&=\mathbb{E}\bigg[\bigg(\frac{f^{(3)}(0)}{3f''(0)\sqrt{n}}Z^3+\frac{f^{(4)}(\theta Z/\sqrt{n})}{12f''(0)n}Z^4\bigg)h'(Z^2)\bigg], \\
|R_4|&\leq\frac{\|h''\|}{2}\mathbb{E}\bigg[\bigg(\frac{f^{(3)}(0)}{3f''(0)\sqrt{n}}Z^3+\frac{f^{(4)}(\theta Z/\sqrt{n})}{12f''(0)n}Z^4\bigg)^2\bigg].
\end{align*}
Since $x^3h'(x^2)$ is an odd function and the standard normal distribution is symmetric about $0$, we have that $\mathbb{E}[Z^3h'(Z^2)]=0$.  Hence
\begin{align*}|R_3|=\frac{1}{12|f''(0)|n}\big|\mathbb{E}[Z^4f^{(4)}(\theta Z/\sqrt{n})h'(Z^2)]\big|\leq \frac{\|h'\|}{12|f''(0)|n}\bigg(3C+\frac{D}{n^{s/2}}\mathbb{E}|Z|^{s+4}\bigg),
\end{align*}
where the calculation used to obtain the inequality is now almost routine to us and we made use of the formula $\mathbb{E}Z^4=3$ and the assumption that $|f^{(4)}|\leq C+D|w|^s$ for all $w\in\mathbb{R}$.  Finally, we bound $R_4$.  In the first step we use the simple bound $(a+b)^2\leq 2(a^2+b^2)$ to obtain a simpler bound and then use similar arguments to those used throughout the proof to obtain the bound
\begin{align*}|R_4|\leq\|h''\|\bigg[\frac{(f^{(3)}(0))^2}{9(f''(0))^2n}\cdot 15+\frac{1}{144(f''(0))^2n^2}\bigg(105C^2+\frac{D^2}{n^{s/2}}\mathbb{E}|Z|^{s+8}\bigg)\bigg],
\end{align*}
where we made use of the formulas $\mathbb{E}Z^6=15$, $\mathbb{E}Z^8=105$. Combining the bounds for $R_1$, $R_3$ and $R_4$ (which together bound $R_2$) now gives the desired bound and completes the proof. \hfill $\Box$

\begin{remark}On examining the proof of Theorem \ref{thmdelta1}, one can see that our assumption that $|f''(w)|\leq C+D|w|^s$ for all $w\in\mathbb{R}$ can be substantially weakened.  Our reason for imposing this condition is that we preferred clarity over the most general result possible, and this assumption ties in quite neatly with the assumption that $|f'(w)|\leq A+B|w|^r$ for all $w\in\mathbb{R}$ (which we must impose in order to apply Theorem \ref{winfirst2}).  The same remark applies to the two bounds of Theorem \ref{thmdelta2}.
\end{remark}

\vspace{-2mm}

\subsubsection{Sequence comparison: the $D_2$ and $D_2^*$ statistics in the general case}\label{d2sec}



Word sequence comparison is of importance to biological sequence comparison.  One way of comparing sequences uses $k$-tuples (a sequence of letters of length $k$), with the intuition being that if two sequences are closely related, we would expect their $k$-tuple content to be similar. A statistic for sequence comparison based on $k$-tuple content, known as the $D_2$ statistic, was suggested by \cite{bla}.  

Suppose that the two sequences, $\mathbf{A}=A_1A_2\ldots A_m$ and $\mathbf{B}=B_1B_2\ldots B_n$, say, are composed of i.i.d$.$ letters that are drawn from a finite alphabet $\mathcal{A}$ of size $d$.  The null hypothesis is typically that the two sequences are independent.  For $a\in \mathcal{A}$ let $p_a$ denote the probability of letter $a$.  For $\mathbf{w}=(w_1,\ldots,w_k)\in\mathcal{A}^k$  let
\[X_{\mathbf{w}}=\sum_{i=1}^{\bar{m}}\mathbf{1}(A_i=w_1,\ldots,A_{i+k-1}=w_k)\] 
count the number of occurrences of $\mathbf{w}$ in $\mathbf{A}$.  Here $\bar{m}=m-k+1$.  Similarly, we let $Y_{\mathbf{w}}$ count the number of occurrences of $\mathbf{w}$ in $\mathbf{B}$, and let $\bar{n}=n-k+1$.  For $\mathbf{w}=w_1\cdots w_k$ denote by $p_{\mathbf{w}}=\prod_{i=1}^kp_{w_i}$ the probability of occurrence of $\mathbf{w}$.  Then the $D_2$ statistic is defined by
\[D_2=\sum_{\mathbf{w}\in\mathcal{A}^k}X_{\mathbf{w}}Y_{\mathbf{w}}.\] 
Due to the complicated dependence structure (for a detailed account see \cite{lothaire}) approximating the asymptotic distribution of $D_2$ is a difficult problem.  However, for certain parameter regimes $D_2$ has been shown to be asymptotically normal and Poisson, with error bounds given \cite{lippert,kant07}. 

An alternative to $D_2$ is the $D_2^*$ statistic \cite{waterman}, given by
\[D_2^*=\sum_{\mathbf{w}\in\mathcal{A}^k}\frac{(X_{\mathbf{w}}-\bar{m}p_{\mathbf{w}})(Y_{\mathbf{w}}-\bar{n}p_{\mathbf{w}})}{\sqrt{\bar{m}\bar{n}}p_{\mathbf{w}}},\]
which was shown by \cite{waterman}, through simulation studies, to outperform $D_2$ in terms of power for detecting the relatedness between the two sequences.  Like $D_2$, the $D_2^*$ statistic has a complicated dependence structure and no quantitative limit theorems have yet been derived.  However, the tools developed in this paper offer the possibility to attack this problem, as we shall now sketch.

The $D_2^*$ statistic is motivated (see \cite[Section 2.2]{waterman}) by estimating the standardised counts
\[X_{\mathbf{w}}^0=\frac{X_{\mathbf{w}}-\bar{m}p_{\mathbf{w}}}{\sqrt{\mathrm{Var}X_{\mathbf{w}}}} \quad \mbox{and} \quad Y_{\mathbf{w}}^0=\frac{Y_{\mathbf{w}}-\bar{n}p_{\mathbf{w}}}{\sqrt{\mathrm{Var}Y_{\mathbf{w}}}}.\]
We have that $\mathrm{Var}(X_{\mathbf{w}})=\bar{m}p_{\mathbf{w}}(1-p_{\mathbf{w}})$.  For relatively rare words $\mathbf{w}$ (rare words will occur provided that $k$ is reasonably large) we have $1-p_{\mathbf{w}}\approx 1$ and so we can approximate $\mathrm{Var}(X_{\mathbf{w}})$ by $\bar{m}p_{\mathbf{w}}$, which is less costly to compute.  Now, $X_{\mathbf{w}}^0$ and $Y_{\mathbf{w}}^0$ are sums of locally dependent random variables and so, under certain parameter regimes, for each $\mathbf{w}$, $X_{\mathbf{w}}^0Y_{\mathbf{w}}^0$ is approximately distributed as the product of two independent standard normal random variables (see \cite{waterman}).  However, the random vectors $\mathbf{X}_{\mathbf{w}}^0=(X_{\mathbf{w}}^0:\mathbf{w}\in\mathcal{A}^k)$ are not independent; in fact, for $\mathbf{d}\in\mathcal{A}^k$, $X_{\mathbf{d}}^0$ is determined by the other $d^k-1$ word counts.  A multivariate normal approximation for $\mathbf{X}_{\mathbf{w}}^0$ (in a certain parameter regime) is given by \cite{h02}, and formulas for the covariance matrix are given in \cite{l90} and \cite{waterman}.  Similar comments apply to $\mathbf{Y}_{\mathbf{w}}^0=(Y_{\mathbf{w}}^0:\mathbf{w}\in\mathcal{A}^k)$.  From the above, we see that $D_2^*$ can be written as 
\[D_2^*=\sum_{\mathbf{w}\in\mathcal{A}^k}a_{\mathbf{w}}X_{\mathbf{w}}^0Y_{\mathbf{w}}^0,\]
where $a_{\mathbf{w}}=\sqrt{\mathrm{Var}(X_{\mathbf{w}})\mathrm{Var}(Y_{\mathbf{w}})/(\bar{m}\bar{n}p_{\mathbf{w}}^2)}$.  Hence, $D_2^*$ is of the form $D_2^*=g(\mathbf{X}_{\mathbf{w}}^0,\mathbf{Y}_{\mathbf{w}}^0)$, where $g(\mathbf{x},\mathbf{y})=\sum_{i=1}^{d^k}a_ix_iy_i$.  The limiting distribution is given by $D_{lim}=g(\Sigma^{1/2}\mathbf{Z}_1,\Sigma^{1/2}\mathbf{Z}_2)$, where the multivariate normal random variables $\Sigma^{1/2}\mathbf{Z}_1$ and $\Sigma^{1/2}\mathbf{Z}_2$ are independent and the covariance matrix $\Sigma$ is given in \cite{l90} and \cite{waterman}.

 It is evident that the problem of bounding the quantity $|\mathbb{E}h(D_2^*)-\mathbb{E}h(D_{lim})|$ shares similarities with the quantitative limit theorems derived in this paper. Here, the function $g:\mathbb{R}^{2(d^k-1)}\rightarrow\mathbb{R}$ is smooth with polynomial growth.  The components of the random vector $(\mathbf{X}_{\mathbf{w}}^0,\mathbf{Y}_{\mathbf{w}}^0)$ are standardised sums of random variables and a multivariate normal approximation is valid (in a particular parameter regime).  Moreover, $g$ is an even function, so a convergence rate of order $m^{-1}+n^{-1}$ may be expected.  However, unlike in this paper, the $D_2^*$ statistic has a dependence structure: there is a local dependence in amongst the summands in the word count statistics $X_{\mathbf{w}}^0$, and a global dependence structure as the word counts statistics $X_{\mathbf{w}}^0$ are dependent themselves.  A work in progress of the author is to extend the theory developed in this paper to treat dependence structures of the type found in the $D_2^*$ statistic, with an application to bounding the quantity $|\mathbb{E}h(D_2^*)-\mathbb{E}h(D_{lim})|$.  Some progress to this goal has been made in \cite{gr16}, in which the theory is extended to locally dependent random variables with application to the rate of convergence of some classical statistics.

\subsubsection{Binary sequence comparison}\label{binsec}

Whilst the problem of bounding bounding the quantity $|\mathbb{E}h(D_2^*)-\mathbb{E}h(D_{lim})|$ for general parameter values is beyond the scope of this paper, thanks to Theorems \ref{winfirst1} and \ref{multieveng} we are, however, able to treat the special case of binary sequence comparison; some details regarding this problem for the $D_2$ statistic are also given in Section 5 of \cite{lippert}.

Consider an alphabet of size $2$ with comparison based on the content of $1$-tuples.  Suppose that the sequences are of length $m$ and $n$, the alphabet is $\{0,1\}$, and $\mathbb{P}(0 \mbox{ appears})=p$ and $\mathbb{P}(1 \mbox{ appears})=q$, with $p+q=1$.  Denoting the number of occurrences of $0$ in the two sequences by $X$ and $Y$, then
\begin{align}D_2^*&=\frac{(X-mp)(Y-np)}{\sqrt{mn}p}+\frac{(m-X-mq)(n-Y-nq)}{\sqrt{mn}q}\nonumber\\
\label{d234}&=\bigg(\frac{X-mp}{\sqrt{mpq}}\bigg)\bigg(\frac{Y-np}{\sqrt{mpq}}\bigg).
\end{align}
By the central limit theorem, $(X-mp)/\sqrt{mpq}$ and $(Y-np)/\sqrt{npq}$ are approximately $N(0,1)$ distributed, and so $D_2^*$ is approximately distributed as $Z_1Z_2$, where $Z_1$ and $Z_2$ are independent $N(0,1)$ random variables.  By Proposition 1.2 of \cite{gaunt vg}, $Z_1Z_2$ is a variance-gamma $\mathrm{VG}(1,0,1,0)$ random variable with density $\frac{1}{\pi}K_0(|x|)$, $x\in\mathbb{R}$, where $K_0(u)=\int_0^\infty (1+t^2)^{-1/2}\cos(ut)\,\mathrm{d}t$ is a modified Bessel function of the second kind. 
In the case $p=q=\frac{1}{2}$, straightforward calculations (see \cite[Section 5]{lippert}) show that
\begin{equation*} \label{cox apple} D2z=\frac{D_2-\mathbb{E}D_2}{\sqrt{\mathrm{Var}(D_2)}} =\bigg(\frac{X-\frac{m}{2}}{\sqrt{\frac{m}{4}}}\bigg)\bigg(\frac{Y-\frac{n}{2}}{\sqrt{\frac{n}{4}}}\bigg). 
\end{equation*}
Note that in this case $D2z=D_2^*$.  We quantify these variance-gamma approximations of $D2z$ and $D_2^*$ in Theorem \ref{d2thm}.  It should be noted that if $p\not=\frac{1}{2}$, then a normal approximation for $D2z$ is more suitable; see again \cite[Section 5]{lippert}.

\begin{theorem}\label{d2thm}Consider the $D_2^*$ statistic, as given by (\ref{d234}), based on 1-tuple content, for i.i.d$.$ binary sequences of lengths $m$ and $n$ drawn from an alphabet $\{0,1\}$ for which $\mathbb{P}(0 \mbox{ appears})=p$, $\mathbb{P}(1 \mbox{ appears})=q$, with $p\in(0,1)$ and $p+q=1$.
Let $V\sim\mathrm{VG}(1,0,1,0)$.  Then, for $h\in C_b^4(\mathbb{R})$,
\begin{equation}\label{ineqwer}|\mathbb{E}h(D_2^*)-\mathbb{E}h(V)|\leq \bigg(179+4411\bigg|\sqrt{\frac{q}{p}}-\sqrt{\frac{p}{q}}\bigg|\bigg)\bigg(\frac{q^4}{p^3}+\frac{p^4}{q^3}\bigg)\bigg(\frac{1}{m}+\frac{1}{n}\bigg)\big(\|h^{(4)}\|+6\|h^{(3)}\|+7\|h''\|+\|h'\|\big).
\end{equation}

Suppose now that $p=q=\frac{1}{2}$.  Then we have the improved bound: for $h\in C_b^3(\mathbb{R})$,
\begin{equation}\label{fiop}|\mathbb{E}h(D_2^*)-\mathbb{E}h(V)|\leq 193\bigg(\frac{1}{m}+\frac{1}{n}\bigg)\big(\|h^{(3)}\|+3\|h''\|+\|h'\|\big).
\end{equation}
As $D2z=D_2^*$ in this case, we have the same bound for $|\mathbb{E}h(D2z)-\mathbb{E}h(V)|$.
\end{theorem}

\begin{proof}Let $\mathbb{I}_i$ and $\mathbb{J}_i$ be the indicator random variables that the letter $0$ occurs at position $i$ in the first and second sequences, respectively.  Then $X=\sum_{i=1}^m\mathbb{I}_i$ and $Y=\sum_{j=1}^n\mathbb{J}_j$, and we may write
\[D_2^*=\bigg(\frac{X-mp}{\sqrt{mpq}}\bigg)\bigg(\frac{Y-np}{\sqrt{npq}}\bigg)=\bigg(\frac{1}{\sqrt{m}}\sum_{i=1}^{m}X_{i}\bigg)\bigg(\frac{1}{\sqrt{n}}\sum_{j=1}^{n}Y_{j}\bigg)=:W_1W_2,\]
where $X_{i}=(\mathbb{I}_i-p)/\sqrt{pq}$ and $Y_{j}=(\mathbb{J}_j-p)/\sqrt{pq}$.  The random variables $X_{1},\ldots,X_m$ and $Y_{1},\ldots,Y_n$ are i.i.d$.$ with zero mean and unit variance.  

We first consider the general $p\in(0,1)$ case and prove inequality (\ref{ineqwer}) using Theorem \ref{multieveng} with $g(u,v)=uv$ (which is an even function).  Now, $\partial_ug=v$, $\partial_vg=u$,
$\partial_{uv}g=1$ and all other partial derivatives are equal to zero.  We can therefore take $P(u,v)=1+u^4+v^4$ as our dominating function.  On applying Theorem \ref{multieveng} with $d=2$, $p=2$, $A=1$, $B=1$ and $r_1=r_2=4$ we obtain a bound for the quantity $|\mathbb{E}h(D_2^*)-\mathbb{E}h(V)|$.  In applying the theorem a number of expectations need to be computed: we have
\begin{align*}|\mathbb{E}X_1^3|=\bigg|\sqrt{\frac{q}{p}}-\sqrt{\frac{p}{q}}\bigg|, \quad \mathbb{E}X_1^8=\frac{q^4}{p^3}+\frac{p^4}{q^3}, \quad \mathbb{E}Z^4=3, \quad \mathbb{E}|Z|^5=8\sqrt{\frac{2}{\pi}},
\end{align*}
where $Z\sim N(0,1)$.  In order to obtain a compact final bound, we use that $1\leq\mathbb{E}|X_1|^a\leq\mathbb{E}X_1^8$ for $2\leq a\leq 8$, which follows from H\"{o}lder's inequality, and that, again by H\"{o}lder's inequality,
\[\mathbb{E}|X_1^3Y_1^4|=\mathbb{E}|X_1|^3\mathbb{E}X_1^4\leq \big(\mathbb{E}X_1^4\big)^2\leq \mathbb{E}X_1^8\]
and
\begin{align*}\mathbb{E}|X_1|^3\mathbb{E}W_1^4\leq \mathbb{E}X_1^4\mathbb{E}W_1^4=\mathbb{E}X_1^4\bigg(\frac{3(n-1)}{n}+\frac{\mathbb{E}X_1^4}{n}\bigg)<3\big(\mathbb{E}X_1^4\big)^2\leq3\mathbb{E}X_1^8,
\end{align*}
as well as the simple inequality $\frac{1}{\sqrt{mn}}\leq\frac{1}{2}\big(\frac{1}{m}+\frac{1}{n}\big)$ and the crude inequalities $m\geq1$, $n\geq1$.  Combining the above formulas and inequalities with the bound of Theorem \ref{multieveng} then yields inequality (\ref{ineqwer}) after rounding the constants up to the nearest integer.

Now suppose that $p=\frac{1}{2}$.  In this case, $\mathbb{E}X_1^3=0$ and we may therefore apply Theorem \ref{winfirst1} to obtain a bound for $|\mathbb{E}h(D_2^*)-\mathbb{E}h(V)|$, which allows us to weaken the class of test functions to $C_b^3(\mathbb{R})$.  Again we have $g(u,v)=uv$, but we now take $P(u,v)=1+|u|^3+|v|^3$ as our dominating function.  We thus apply Theorem \ref{winfirst1} with $d=2$, $p=3$, $A=1$, $B=1$ and $r_1=r_2=3$.   In obtaining our bound we use that $\mathbb{E}|X_1|^a=1$ for all $a>0$, and that, by H\"older's inequality,
\begin{equation*}\mathbb{E}|W_1|^3\leq\big(\mathbb{E}W_1^4)^{3/4}=\bigg(\frac{3(n-1)}{n}+\frac{\mathbb{E}X_1^4}{n}\bigg)^{3/4}=\bigg(\frac{3n-2}{n}\bigg)^{3/4}<3^{3/4}.
\end{equation*}
To obtain the compact bound (\ref{fiop}) for $|\mathbb{E}h(D_2^*)-\mathbb{E}h(V)|$ we then proceed as we did in obtaining our bound for the general $p\in(0,1)$ case.
\end{proof}


\begin{remark}





For the case $p=\frac{1}{2}$, a bound of order $m^{-1}+n^{-1}$ for the quantity $|\mathbb{E}h(D2z)-\mathbb{E}h(V)|$ (for bounded test functions $h$ whose first three derivatives are also bounded) is given in \cite{gaunt vg} and Wasserstein and Kolmogorov distance bounds between the distributions of $D2z$ and $V$ with slower convergence rates are given in \cite{gaunt vg2}.  Our bound (\ref{fiop}) improves on that of \cite{gaunt vg} by having smaller constants and not requiring $h$ to be bounded, and our bound \ref{ineqwer} is the first in the literature to treat the general $p\in(0,1)$ case. 

Let us discuss the bound (\ref{ineqwer}) in more detail.  The bound increases as the quantity $|p-q|$ increases and, for fixed $m$ and $n$, blows up in the limits $p\rightarrow0$ and $p\rightarrow1$, but the bound still tends to 0 in the limit $np^{7/2}(1-p)^{7/2}\rightarrow\infty$.  However, $(X-mp)/\sqrt{mpq}$ and $(Y-np)/\sqrt{npq}$ convergence in distribution to the standard normal distribution in the limits $mp(1-p)\rightarrow\infty$ and $np(1-p)\rightarrow\infty$, respectively, so, by the continuous mapping theorem, the $\mathrm{VG}(1,0,1,0)$  approximation of $D_2^*$ is also valid if $mp(1-p)\gg1$ and $np(1-p)\gg1$.  The reason that our bound does not have an optimal dependence on the parameter $p$ is an artefact of the fact that we applied the general bound of Theorem \ref{multieveng} that is not optimised for this particular application.  Centering indicator random variables (as we must do to apply Theorem \ref{multieveng}) is not very efficient in terms of obtaining bounds with optimal dependence on the parameter $p$, and we refer the reader to the proof of Theorem 4.1 of \cite{gaunt chi square}, particularly the Taylor expansion on p$.$ 739, for a specialised  method for dealing with a statistic that is expressed in terms of indicator random variables that leads to an optimal dependence on all parameters.  
\end{remark}


\section*{Acknowledgements}
During the course of this research, the author was supported by an EPSRC DPhil Studentship, an EPSRC Doctoral Prize and  EPSRC grant EP/K032402/1.  The author is currently supported by a Dame Kathleen Ollerenshaw Research Fellowship.  The author would like to thank Gesine Reinert for helpful discussions.
Finally, the author would like to thank the anonymous referees for their constructive comments and suggestions that have led to an improved paper.

\end{document}